\documentclass[11pt, leqno]{amsart}
\usepackage{amssymb,latexsym}
\usepackage[all]{xy}
\newtheorem{thm}{Theorem}[section]
\newtheorem*{thm*}{Theorem}
\newtheorem{proposition}[thm]{Proposition}
\newtheorem{lemma}[thm]{Lemma}
\newtheorem{cor}[thm]{Corollary}
\theoremstyle{definition}

\newtheorem{rmk}[thm]{Remark}

\numberwithin{equation}{section}
 
\newtheorem{notation}[thm]{Notation}  

\newcommand{\boxtensor}{{\Box\kern-9.03pt\raise1.42pt\hbox{$\times$}}}

\newcommand{\F}{{\mathbb F}}
\newcommand{\N}{{\mathbb N}}

\newcommand{\Z}{{\mathbb Z}}

\newcommand{\sB}{{\mathcal B}}

\renewcommand{\tilde}{\widetilde}
\renewcommand{\hat}{\widehat} 
\numberwithin{equation}{section}
\newcounter{elno}                

\newcounter{example}[section] 
\def\theexample{\thesection.\arabic{example}}

\begin{document}
\title{Rings of invariants for three dimensional modular representations}
\author{J\"urgen Herzog and Vijaylaxmi Trivedi}

\begin{abstract}{Let $p>3$ be a prime number. We compute the rings of invariants of 
the elementary abelian $p$-group $(\Z/p\Z)^r$ for $3$-dimensional generic 
representations. Furthermore we show that these rings of invariants are complete 
intersection rings  with embedding dimension $\lceil r/2\rceil +3$.

This proves a conjecture of Campbell, Shank and Wehlau in [CSW], which they  proved  
for $r=3$,  and was later proved for $r=4$ by Pierron and Shank.
}\end{abstract}


\thanks{The second author  acknowledges  support of the Department of Atomic
Energy, India under project number RTI4001.}

\address{J\"urgen Herzog, Fakult\"at f\"ur Mathematik, Universit\"at Duisburg-Essen, 
45117 Essen, Germany} \email{juergen.herzog@uni-essen.de}

\address{School of Mathematics, Tata Institute of
Fundamental Research,
Homi Bhabha Road, Mumbai-400005, India}
\email{vija@math.tifr.res.in}

\date{}
\maketitle

\section{Introduction}

Let $\F$  a field of characteristic $p>2$ and let $G$ be a finite abelian group. 
Furthermore, let $V$ be an $n$-dimensional representation of $G$ over the field $\F$.
Then by choosing a  basis $\{x_1, \ldots, x_n\}$ for $V^\ast$ 
we have the canonical right action of $G$ on the polynomial ring $R = \F[x_1, \ldots, x_n]$. 
It is a classical problem of (algebraic) invariant theory to determine the structure 
of the ring of invariants

$$R^G = \F[x_1, \ldots, x_n]^G = \{f\in \F[x_1,\ldots, x_n]\mid f\cdot g = f\quad \forall\quad
 g\in G\}.$$

If the representation is non modular, that is, if $p\nmid |G|$  
then the group algebra ${\F}[G]$ is a semisimple ring and  
the ring $R^G$ is 
a direct summand of the polynomial ring $R$.

However the theory of modular representations turns out to  be much 
more complicated
 even for elementary abelian
$p$-groups $G= (\Z/p\Z)^r$, (see for example  [B, \S 4.4]). In the case $r=1$, 
that is  for the cyclic group $\Z/p\Z$, the theory has been studied by 
Dickson~\cite{D}  and later by Wehlau~\cite{We}.

Campbell, Shank and Wehlau initiated in their paper \cite{CSW} the study of 
the rings of invariants of modular representations  of  such 
groups for  $p>2$.  They considered the rings 
of invariants  of the group  $G = (\Z/p\Z)^r$  of rank $r\geq 2$ 
for a  modular representation of dimension $n\leq 3$.  
 More precisely they showed 

\begin{enumerate}
\item if $n=2$ then  the ring 
of invariants $\F[X, Y]^G$ is  a  polynomial algebra in two variables over $\F$.

\item  If  $n=3$ and the representation is of

\begin{enumerate}
\item  type $(2,1)$:
 $\dim_\F(V^G) = 2$ and $\dim_\F((V/V^G)^G) = 1$. Then 
$\F[X, Y, Z]^G$ is a polynomial ring;

\item type $(1,2)$:  $\dim_F(V^G) = 1$ and $\dim_\F((V/V^G)^G) = 2$.  Then
  $\F[X, Y, Z]^G$ is again a polynomial algebra;

\item type(1,1,1): $\dim_F(V^G) = 1$ and $\dim_\F((V/V^G)^G) = 1$.   Then
  $\F[X, Y, Z]^G$ is a complete intersection ring provided $r\leq 3$.
\end{enumerate}
\end{enumerate}
 In a 
subsequent paper  Pierron and Shank treated the case $r=4$, see \cite{PS}.

In both situations ({\em i.e.}, $r = 3$ and $r=4$)
 the authors first construct a SAGBI basis for the
generic representation (definition is recalled below) of type $(1,1,1)$. Then, 
by specialization and some
polynomial conditions, they 
construct a SAGBI basis  for general representations of type $(1,1,1)$.

First the authors in \cite{CSW} observe   that every  representation $V$ of type
$(1,1,1)$ is equivalent to $V_M$, where 
$M = \left[c_{ij}\right]_{2\times r}\in \F^{2\times r}$ and  where
$V_M$ is a three dimensional  representation, whose action  of $G$ with respect to the 
basis $\{X,Y,Z\}$ for $V_M^\ast$  as follows: 

$$X\cdot e_j = X, \quad Y\cdot e_j = Y+c_{1j}X, \quad Z\cdot e_j =
Z+2c_{1j}Y+(c_{1j}^2+c_{2j})X. 
$$
Here $e_1,\ldots,e_r$ are the canonical generators  of $G=(\Z/p\Z)^r$.

In other words, for the representation  $G\longrightarrow GL_3(\F)$, 
the $e_j$ can be represented by the matrix
$$ e_j \to \left[\begin{matrix} 1 & 2c_{1j} & c_{1j}^2+c_{2j}\\
0 & 1 & c_{1j}\\
0 & 0 & 1\end{matrix}\right].$$

The authors compute the ring of invariants $\F[X, Y, Z]^G$ 
for these representations by  considering them
as a specialization  of the generic representation which is given as follows: 
Consider a polynomial algebra $\F_p[x_{ij}]:= \F_p[x_{1j}, x_{2j}\mid j = 1, 2, 
\ldots, r]$ in $2r$ indeterminates,
and its quotient field ${\mathcal{K}} = \F_p(x_{ij})$.
Let  $G= (\Z/p\Z)^r$ be the elementary $p$-group with  the canonical generators
$e_1, \ldots, e_r$,  and let $V$ be a three dimensional ${\mathcal{K}}$-vector space with the basis
$\{X, Y, Z\}$ for $V^\ast$.
Let the action of $G$ be given by the group homomorphism
$\rho: G \longrightarrow GL_3({\mathcal{K}})$
\begin{equation*}e_j \to \left[\begin{matrix} 1 & 2x_{1j} & x_{1j}^2+x_{2j}\\
0 & 1 & x_{1j}\\
0 & 0 & 1\end{matrix}\right],\end{equation*}
therefore, for any given $j$
\begin{equation}\label{e1}X\cdot e_j = X, \quad Y\cdot e_j = x_{1j}X+Y\quad\mbox{and}\quad  
Z\cdot e_j = (x_{1j}^2+x_{2j})X+2x_{1j}Y+Z.\end{equation}

\vspace{5pt}

Then there exist homogeneous polynomials $f_1, f_2, f_3, N_G(Z):= \prod_{g\in G}Z\cdot g \in {\mathcal{K}}[V]^G$
such that
\begin{enumerate}
\item if $r=2$ then $\{X, f_1, f_2, N_G(Z)\}$ is a SAGBI basis for
${\mathcal{K}}[V]^G$ (Theorem~6.2 of \cite{CSW}).
\item if $r=3$ then $\{X, f_1, f_2, f_3, N_G(Z)\}$ is a SAGBI basis for
${\mathcal{K}}[V]^G$ (Theorem~7.4 of \cite{CSW}).
 \item If $r=4$ then $\{X, f_1, f_2, f_3, N_G(Z)\}$ is a SAGBI basis for
${\mathcal{K}}[V]^G$ (Theorem~2.3 of \cite{PS}). 
\end{enumerate}

Based on the numerous calculations using Magma \cite{BCP}, 
the authors  in  \cite{CSW}
made  the following conjecture for general $r$.
Here recall that  the polynomial ring ${\mathcal{K}}[X, Y, Z]$ is given the  graded
reverse lexicographic order with $X<Y<Z$, and  $LM(f)$ denotes the 
leading monomial of $f$.

\vspace{5pt}

\noindent{\bf Conjecture}~(\cite{CSW})
{\em Let $V$ be  the  generic three dimensional representation of the group 
$G= (\Z/p\Z)^r$ over 
the field
 ${\mathcal{K}}:= \F_p(\{x_{1j},x_{2j}\mid j=1,2, \ldots, r\})$.
Let $s=\lceil r/2 \rceil$. 

Then the ring of invariant ${\mathcal{K}}[V]^G$ is a complete 
intersection ring with embedding dimension $s+3$. 
Furthermore, there exists a SAGBI basis $\{X, f_1, \ldots, f_{s+1}, N_G(Z)\}$, where 
$N_G(Z) = \prod_{g\in G}Z\cdot g$,
such that:

\begin{enumerate}

\item if $r=2s$ then 
$$LM(f_1) = Y^{p^{s}},  \quad \mbox{and}\quad 
LM(f_i) = Y^{p^{s+i-2}+2p^{s-i+1}},$$
for $i>1$ and the relations are constructed by subducting the 
t\^{e}te-\`{a}-t\^{e}tes
$$(f_2^p, f_1^{p+2}), \quad \mbox{and}\quad
 (f_i^p, f_{i-1}f_1^{(p^2-1)p^{i-3}}), ~~~\mbox{for}~~i>2,$$
\item if $r=2s-1$ then 
$$LM(f_1) = Y^{2p^{s-1}}, \quad LM(f_2) = Y^{p^s}, \quad \mbox{and}\quad 
LM(f_i) = Y^{p^{s+i-3}+2p^{s-i+1}},$$
for $i>2$ and the relations are constructed by subducting the t\^{e}te-\`{a}-t\^{e}tes
$$(f_1^p, f_2^2), \quad (f_3^p, f_1f_2^p), \quad \mbox{and}\quad
 (f_i^p, f_{i-1}f_2^{(p^2-1)p^{i-4}}), ~~~\mbox{for}~~i>3.$$
\end{enumerate}}

In this paper we prove the above conjecture, except for the particular case $G = (\Z/3\Z)^{2s}$.
We follow  the strategy given in 
\cite{CSW}, in particular,  SAGBI bases are used here to compute the ring of invariants  
by means  of the   so-called SAGBI/Divide-by-$x$ algorithm,  which has been 
introduced by Campbell et al.   Also the explicit description of the leading monomials
as stated in the conjecture   works here as a guiding light for the proof.

It was expected that as $r$ increases the calculation will become exponentially 
complex as seen, {\em e.g.}, in  $r=2$ versus  $r=3$ versus $r=4$ cases.
Here one of the main points  is to  write down a closed formula 
((\ref{f_{j+3}}) and (\ref{f_{j+3}e}))
for the coefficients of 
  $f_j$, mod the ideal $X^{2p^{s-(j-1)}}(X, Y)$,
in terms of the explicit minors in $\F_p[x_{ij}]$. This allows us to 
inductively   construct $f_3, \ldots, f_{s+1}$, for all $r$  in a 
straightforward  way. 

We expect that this description of the coefficients 
 would make it possible  to compute the invariants 
and stratify the space of 
all three dimensional  
representations of $(\Z/p\Z)^r$ as   was done  for $r=3$ in \cite{CSW} and for 
$r=4$ in \cite{PS}.

Now to construct the relations for ${\mathcal{K}}[X,Y,Z]^G$ we use a lemma due to 
Watanabe \cite{Wa}: If  $H$ is a numerical semigroup such that the semigroup ring 
${\mathcal{K}}[H]$ is a complete intersection, then so is ${\mathcal{K}}[H']$ provided $H'$ is obtained from
$H$ by {\em gluing}. Moreover the relations for ${\mathcal{K}}[H']$ are given in terms of the
{\em gluing pair} of elements.

Here, by such gluing operations, we construct  a sequence  (depending on the parity of $r$) 
of numerical semigroups, where  the 
 first semigroup is $\N$ and therefore
the first semigroup ring is a polynomial ring and the final semigroup ring is
${\mathcal{K}}[LT(f_1), \ldots, LT(f_{s+1})]$ with explicit relations.

This, along with the
well known  result of Robbiano-Sweedler~\cite{RS}, implies that to verify 
the conjecture 
it is enough to check that  each of the 
 t\^{e}te-\`{a}-t\^{e}tes,  as given in the conjecture of \cite{CSW},
 subducts to $0$, which now easily follows from the construction of the $f_i$.

The layout of the paper is as follows.

In Section~2 we recall the basic results about SAGBI basis. In order to make the paper 
self contained we  recall some results from  [CSW], and  
a result about the Pl\"ucker relations for the minors from [LR] 
(also see [BH]).

We also  describe the `seed' generators $f_1$ and $f_2$ which are a normalization of 
 the `seed' generators constructed in [CSW].

In Section~3 we consider the case when the rank of  $G$ is $r = 2s-1$ and $p>2$.
Here we  construct a set of elements $f_3, \ldots, f_{s+1}, f_{s+2}$ in the 
invariant ring 
${\mathcal{K}}[X, Y,Z]^G$, where $\{f_1, \ldots, f_{s+1}\}$  have the  same leading terms  
as in the conjecture and the $LM(f_{s+2}) = LM(N_G(Z))$.
We also explicitly write the formula for the elements $f_3,\ldots, f_{s+1}, f_{s+2}$.

In Section~4 we do the similar construction for $r=2s$ and for $p>3$. Here 
the element $f_3$  is a normalization of the element $f_3$ given in [PS].

In Section~5  we prove the main results namely 
Theorem~\ref{t1} and Theorem~\ref{t1e}, in particular
we show that the  set $\{X, f_1, \ldots, f_{s+1}, N_G(Z)\}$ is a
 SAGBI basis for 
${\mathcal{K}}[X, Y,Z]^G$ and the ring of invariants  is a complete intersection 
ring.

\vspace{5pt}

The second author  would like to thank Prof. M. Levine for the 
 support
 to facilitate  her visit to Universit\"at Duisburg-Essen, where a part of 
the work was done. She is also thankful for his insightful comments during and after  her talk on the draft version of the work.
She  would also like to thank Dr. Jai Laxmi for  bringing the paper 
[CSW] to her notice during a  discussion on 
Hilbert-Kunz multiplicities of ring of invariants of modular representations.

We are grateful to 
Prof. R.J. Shank for carefully reading the manuscript and for making numerous
suggestions which 
improved the exposition
tremendously.

We are also grateful to the referee for a meticulous job, which 
has  improved the clarity and the exposition, and corrected 
a number of minor typographical errors.

\section{Preliminaries}

We briefly recall some basic facts about the representation theory of finite 
abelian groups. 

Let $\F$ be a field of characteristic $p>0$. 
Let $G$ be a finite abelian group, and 
let $V$ be an $n$-dimensional representation of $G$ over the field $\F$.
If we consider $V$ as a left $\F[G]$-module and $V^*$ as a right $\F[G]$-module
and denote the symmetric algebra on $V^*$ by $\F[V]$,  then
the action of $G$ on $V^*$ induces a natural degree preserving
algebra automorphism on $\F[V]$. 

This means for a given basis 
$\{x_1, x_2, \ldots, x_n\}$ of $V^*$, 
the $\F$-algebra $\F[V]$ may be identified with the polynomial ring
$\F[x_1, x_2, \ldots, x_n]$ and 
giving a  representation $V$ of $G$ is equivalent to giving 
 a group homomorphism $\sigma:G\longrightarrow GL_n(\F)$.

A different choice of basis of $V^*$ gives a different but conjugate group homomorphisms.
 Hence
 the ring of invariants 
$$\F[x_1, \ldots, x_n]^G = \{f\in \F[x_1,\ldots, x_n]\mid f\cdot g = f\quad \forall\quad
 g\in G\}$$
is well defined up to a ring isomorphism.

Now 
we recall the notion of SAGBI basis which was  introduced independently by
Robbiano-Sweedler~\cite{RS} and Kapur-Madlener~\cite{KM} (see also \cite{EH}
for an exposition).

Let $S = K[X_1, \ldots, X_n]$ be a polynomial ring over a field $K$, and
let $A\subseteq  S$ be a finitely generated subalgebra. We fix a term
order $<$ for the monomials in $S$, and let  $LM(A)$  be
the $K$-subalgebra of $S$ generated by the initial monomials $LM(a)$,
$a\in A$. The algebra
$LM(A)$ is called the lead monomial  algebra of $A$ (with respect to the 
term order $<$). It is clear that $LM(A)$  is the semigroup algebra
of a suitable additive  sub semigroup $H$  of $\mathbb{N}$.

A subset $\sB\subset A$ is called a SAGBI basis of $A$ if 
$LM(A)= K[\{LM(g): g\in \sB\}]$.
It is easy to check that if $\sB$ is a SAGBI basis for $A$,  then it 
generates $A$ as
 $K$-algebra.

 Let 
$\sB=\{f_1,\ldots,f_m\}$ be a finite set of polynomials in $S$, and let
$A=K[f_1,\ldots,f_m]$. We recall  a criterion for SAGBI bases due to
Robbiano and Sweedler \cite{RS}. Let $f_i = a_i\tilde{f}_i$, where
$a_i\in K\setminus \{0\}$ and 
$\tilde{f}_i$ is monic for $i=1,\ldots,m$. It is easy to see that $\sB$ is
a SAGBI basis for  $A$ if and only if $\{\tilde{f}_1,\ldots,\tilde{f}_m\}$
is a SAGBI basis for $A$. Thus we may assume that all $f_i\in\sB$ have
leading coefficient $1$.

\begin{thm}\label{RSKM}
Let  $R=K[t_1,\ldots,t_m]$ be a polynomial ring, and let 
$$\varphi : R=K[t_1,\ldots,t_m]\to A\quad\mbox{given by}\quad 
\varphi(t_i) = f_i\quad\mbox{for}\quad i=1,\ldots,m,$$ 
be the surjective $K$-algebra 
homomorphism. Furthermore, let  
$$\phi:R\longrightarrow K[LM(f_1),\ldots, LM(f_m)] = B
\quad\mbox{given by}\quad 
\phi(t_i) = LM(f_i)$$ 
be the $K$-algebra homomorphism.

The ideal $\ker(\phi)$ is generated by binomials, since $B$ is a toric 
$K$-algebra. For an integer vector ${\bf a}=(a_1,\ldots,a_m)$ we set 
${\bf t}^{\bf a}=t_1^{a_1}t_2^{a_2}\cdots t_m^{a_m}$.

 Let  $\{{\bf t}^{{\bf a}_1}-{\bf t}^{{\bf b}_1},\ldots,
{\bf t}^{{\bf a}_r}-{\bf t}^{{\bf b}_r}\}$  be a  set of binomial generators 
of the toric ideal $\ker(\phi)$.
 Then $f_1,\ldots,f_m$ is a SAGBI basis of $A$, if and only if, for each $i$
 the pair $({\bf f}^{{\bf a}_j}, {\bf f}^{{\bf b}_j})$
 subducts to $0$, {\em i.e.},
 there 
exist elements $c_{j, {\bf a}}\in k$ such that

\begin{eqnarray}\label{subduct1}
\label{lifting}
{\bf f}^{{\bf a}_j}-{\bf f}^{{\bf b}_j}=
\sum_{{\bf a}}c_{j, {\bf a}}{\bf f}^{{\bf a}}\quad \text{with}
\quad LM({\bf f}^{\bf a})< LM({\bf f}^{{\bf a}_j})\quad \text{for all}
\quad {\bf a}_j, {\bf b}_j,
\end{eqnarray}
where  ${\bf f}^{{\bf a}}=f_1^{a_1}\cdots f_m^{a_m}$ for ${\bf a}
=(a_1,\ldots,a_m)$.

If these equivalent conditions hold, then the polynomials
\[
G_j(t_1,\ldots,t_m)={\bf t}^{{\bf a}_j}-{\bf t}^{{\bf  b}_j}-
\sum_{{\bf a}}c_{j, {\bf a}}{\bf t}^{{\bf a}}, \quad j=1,\ldots, r,
\]
generate $\ker(\varphi)$.
\end{thm}

 The pairs  $({\bf f}^{{\bf a}_j},{\bf f}^{{\bf b}_j} )$ are 
called  t\^{e}te-\'{a}-t\^{e}tes.  One says that the    
t\^{e}te-\'{a}-t\^{e}te  $({\bf f}^{{\bf a}_j},{\bf f}^{{\bf b}_j}) $ 
subducts to $0$, if equation (\ref{subduct1}) holds.

\medskip

We recall the following two theorems from \cite{CSW}.

\begin{thm}\label{CSW}(\mbox{Theorem~1.1,~\cite{CSW}}).\quad  Let $K$ be a field and $S= K[X, Y_1, \ldots, Y_n]$ 
be the polynomial ring. Let $B\subset S$  be a $K$-algebra generated by homogeneous 
polynomials and assume that $X\in B$. Let $<$ be the graded reverse lexicographic 
order induced by $X<Y_1 < \cdots < Y_n$, and let $f_1, \ldots, f_l$ be 
homogeneous polynomials in $B$ such that $LM(f_i) \in 
K[Y_1, \ldots, Y_n]$ for $i=1, \ldots, l$. Let $A=K[X, f_1, \ldots, f_l]$ and 
suppose that $\sB = \{X, f_1, \ldots, f_l\}$ is a SAGBI basis for $A$. 
Suppose further that 
\begin{enumerate}\item $A[X^{-1}] = B[X^{-1}]$;
\item $S$ is an integral extension of $A$.
\end{enumerate}
Then $A= B$, and consequently $\sB$ is  a  SAGBI basis of $B$.
\end{thm}

Let ${\mathcal{K}} = \F_p(x_{ij})$ and let the group  $G= (\Z/p\Z)^r$ act
on  ${\mathcal{K}}[X, Y, Z]$ as in (\ref{e1}). We assume $r\geq 3$.
Now we recall the description, from \cite{CSW}, of the elements $g_1$ and $g_2\in 
{\mathcal{K}}[X, Y, Z]^G$. For our purpose we would normalize them to construct $f_1, f_2$ in 
${\mathcal{K}}[X, Y, Z]^G$.

Consider  a   $(2r+2)\times r$ matrix with entries in $k$ 
$$\Gamma:= \left[\begin{matrix} x_{11} & x_{12} & \cdots & x_{1r}\\
x_{21} & x_{22} & \cdots & x_{2r}\\
x^p_{11} & x^p_{12} & \cdots & x^p_{1r}\\
x^p_{21} & x^p_{22} & \cdots & x^p_{2r}\\
& & \vdots & \\
x^{p^r}_{11} & x^{p^r}_{12} & \cdots & x^{p^r}_{1r}\\
x^{p^r}_{21} & x^{p^r}_{22} & \cdots & x^{p^r}_{2r}
\end{matrix}\right].$$

For a subsequence $I= (i_1, \ldots, i_r)$ of $(1, 2,\ldots, 2r+2)$, let 
$\nu_I\in {\mathcal{K}}$ be the associated minor of $\Gamma$.
 Consider a $(2r+2)\times (r+1)$ matrix ${\tilde \Gamma}$
with entries in the extended ring
${\mathcal{K}}[X, Y, Z][X^{-1}]$. This is 
 constructed  by  a 
column augmention of  the matrix $\Gamma$, where $\Delta = Y^2-XZ$,

 $${\tilde \Gamma} := \left[\begin{matrix}x_{11} & x_{12} & \cdots & x_{1r} & Y/X\\
x_{21} & x_{22} & \cdots & x_{2r} & -\Delta/X^2\\
x^p_{11} & x^p_{12} & \cdots & x^p_{1r} & (Y/X)^p\\
x^p_{21} & x^p_{22} & \cdots & x^p_{2r} & (-\Delta/X^2)^p\\
& & \vdots & \\
x^{p^r}_{11} & x^{p^r}_{12} & \cdots & x^{p^r}_{1r}& (Y/X)^{p^r}\\
x^{p^r}_{21} & x^{p^r}_{22} & \cdots & x^{p^r}_{2r} & (-\Delta/X^2)^{p^r}
\end{matrix}\right].$$

For a subsequence $J= (j_1, \ldots, j_{r+1})$ of $(1, 2, \ldots, 2r+2)$, let 
${\tilde f_J}\in {\mathcal{K}}[X, Y, Z][X^{-1}]$ denote the associated $(r+1)\times 
(r+1)$ minors of ${\tilde \Gamma}$. Let $f_J$ denote the element of 
${\mathcal{K}}[X, Y, Z]$ constructed by minimally clearing the denominator of ${\tilde f_J}$.

Now for  $\Delta_j:= e_j-1\in {\mathcal{K}}[G]$, where $e_1, \ldots, e_r$ are the canonical generators of $G = (\Z/p\Z)^r$ we have 
$$Y\cdot \Delta_j =  Y\cdot e_j - Y =  Xx_{1j}\quad \mbox{and}\quad 
\Delta\cdot \Delta_j = (Y^2-XZ)\cdot e_j - (Y^2-XZ) = -X^2x_{2j}.$$
Therefore 
if ${\bf v}$ denotes the last column of ${\tilde \Gamma}$, then 
$${\bf v}\cdot \Delta_j = [\Gamma_{1j}, \Gamma_{2j}, \ldots, \Gamma_{(2r+1)j}]^T,$$
the $j^{th}$ column of $\Gamma$. Thus ${\tilde f_J}\cdot \Delta_j = 0$ 
 means   
 ${\tilde f_J}\in {\mathcal{K}}[X, Y, Z]^G[X^{-1}]$ and
$f_J\in {\mathcal{K}}[X, Y, Z]^G$.

\begin{enumerate}
\item If $r = 2s-1\geq 3$ is an odd integer then 
 \begin{equation}\label{g1e} 
{g_1} := f_{(1, 2, \ldots, r, r+1)}\quad\mbox{and}\quad 
{g_2}:= f_{(1,2, \ldots, r, r+2)}.\end{equation}

\item If $r = 2s\geq 4$ is an even integer and $p>3$ is a prime number.
Then $g_1$ and $g_2$ are given as follows.
Let $g_1 := f_{(1, 2, \ldots, r, r+1)}$ and 
${\bar g_2}:= f_{(1,2, \ldots, r, r+2)}$.
Let  $c_0 = \nu_{1,2,\ldots, r} $ and 
$c_1 = \nu_{1,2,\ldots, {\hat{r}}, r+1}$ denote 
 the  $r\times r $ minors of $\Gamma$.
Then we define 
\begin{equation}\label{g2e}g_2 = 
\frac{c_0{\bar g_2}+g_1^2}{c_0c_1X^{p^s-2p^{s-1}}}.
\end{equation}
\end{enumerate}

Note that in \cite{CSW} $g_i$ are denoted as $f_i$ and $g_2 = (1/c_0c_1)f_2$.

\begin{thm}\label{CSW2}(\mbox{Theorem~5.2. ~\cite{CSW}}).\quad For $g_1$, 
$g_2$ as in (\ref{g1e}) and 
(\ref{g2e}) and for the action of the group $G = (\Z/p\Z)^r$ on ${\mathcal{K}}[X, Y, Z]$ as 
in (\ref{e1}), we have
${\mathcal{K}}[X, Y, Z]^G[X^{-1}] = {\mathcal{K}}[X, g_1, g_2][X^{-1}]$.\end{thm}

The following lemma follows from Theorem~3.1  of [SW].
Here we give a self contained proof which is along the same lines as in
  the second part of the 
 proof of Theorem~3.1 of [SW].

\begin{lemma}\label{l1} Let $B = {\mathcal{K}}[X, Y, Z]^G$. If $f_{s+2} = 
Z^{p^r}+Xp_1+Y p_2$ is a homogeneous polynomial  in $B$, where 
$p_1, p_2\in {\mathcal{K}}[X,Y,Z]$. Then $B$  is integral over the ring
$A = {\mathcal{K}}[X, g_1, g_2, f_{s+2}]$. \end{lemma}

\begin{proof}Let us denote  $l_1 = X$, $l_2 = g_1$ and $l_3 = f_{s+2}$ 
and let $A' = {\mathcal{K}}[l_1, l_2, l_3]$.
Since  $A'\subseteq 
A\subseteq  B\subseteq S = {\mathcal{K}}[X, Y, Z]$,  it is enough to prove that 
$S$ is integral over $A'$.

Let ${\bf m}_{A'} = \sum_{i=1}^3l_iA'$
denote the maximal ideal of $A'$. Note that 
$l_2 = \lambda Y^{2p^{s-1}}+X{p_3}$ and 
$l_3 = Z^{p^r}+X{p_1}+Y{p_2}$, where 
${p_i}\in {\mathcal{K}}[X,Y,Z]$ and $\lambda\in {\mathcal{K}}\setminus \{0\}$. 
Therefore ${\bf m}_{A'}S \supseteq  (X, Y^{2p^{s-1}}, Z^{p^{r+s}})$. 
In particular
$S/{\bf m}_{A'}S$ is a finite dimensional ${\mathcal{K}}$-vector space.

 Let $\{h_1, \ldots, h_m\}$ be a set of homogeneous 
elements in $S$ which 
generate the ${\mathcal{K}}$-vector space $S/{\bf m}_{A'}S$ and let 
$H = \sum_{i=1}^mh_iA$.
 Then $S = H + {\bf m}_{A'}S$. Now we show that $S = H$.

We note that 
 $S/H$ is a $\N$-graded $A'$-module. If it is nonzero then
we can choose an element $r\in S\setminus H$ 
which is a homogeneous element and of least possible degree.
We can write $r=\sum_{i}l_is_i$ mod $H$ such that $s_i$ are homogeneous
elements in $S$ and $l_is_i\not\in H$.
Now for some $i$, $\deg~l_is_i = \deg~r$. But then
$s_i\in S\setminus  H$ such that $\deg~s_i < \deg~r$ which is a contradiction.
\end{proof}

\begin{rmk}\label{r1} Let $g_1$ and $g_2$ be the elements in $B={\mathcal{K}}[X, Y, Z]^G$
as  given  in (\ref{g1e}) and (\ref{g2e}). 
In the next two sections (for $r$ odd and for $r$ even) we construct a set of elements  
$\sB_1 = \{X, f_1, f_2, \ldots, f_{s+1}, f_{s+2}\}$  in $B$, where  $f_{s+2} = Z^{p^r}+X p_1+Y p_2$ is a homogeneous polynomial, and ${\mathcal{K}}[X, f_1, f_2] = 
{\mathcal{K}}[X, g_1, g_2]$. 
In Section~5 we
show that $\sB_1$ is a SAGBI basis for $A = {\mathcal{K}}[X, f_1, f_2
\ldots, f_{s+1}, f_{s+2}] \subseteq {\mathcal{K}}[X,Y,Z]^G$. This, by Theorem~\ref{CSW}, will imply that
$\sB_1$ is a SAGBI basis for 
${\mathcal{K}}[X,Y,Z]^G$ and ${\mathcal{K}}[X, Y, Z]^G = A$.

On the other hand,  by the choice of term order we have
$LT(N_G(z)) = LT(f_{s+2})$, where $N_G(Z) = \prod_{g\in G}Z\cdot g\in B$.
Hence we have  the  set $\{X, f_1, f_2, \ldots, f_{s+1}, N_G(z)\}$ 
as  a SAGBI basis for $B$.
\end{rmk}

We recall the following theorem from Lakshmibai-Raghavan~[LR, $\S$ 4.1.3] 
or Bruns-Herzog~[BH, Lemma~7.2.3],  which 
gives a crucial family of Pl\"ucker relations for the minors of $\Gamma$ 
and ${\tilde \Gamma}$.
\begin{thm}\label{PS} Let $N$ be an $n\times m$ matrix with $n>m$. Denote by 
$\nu_{i_1, \ldots, i_m}$ the minor of $N$ determined by the rows 
$i_1, \ldots, i_m$. 
For sequences $(i_1, \ldots, i_{m-1})$ and $(j_1, \ldots, j_{m+1})$, we have 
the following Pl\"ucker relation
$$\sum_{a=1}^{m+1}(-1)^a\nu_{i_1, \ldots, i_{m-1}, j_a}\cdot 
\nu_{j_1, \ldots, j_{a-1}, j_{a+1}, \ldots, j_{m+1}} = 0.$$
\end{thm}

\section{When $r=2s-1$ is odd }

Throughout this section ${\mathcal{K}} = \F_p(x_{ij})$ and   $G= (\Z/p\Z)^r$ acts
on  ${\mathcal{K}}[X, Y, Z]$ as in (\ref{e1}). Here we deal with the case when 
$r= 2s-1\geq 3$ is an odd 
integer and $p>2$ is a prime number. The proof for $r=3$ is same as in [CSW].

In Proposition~\ref{odd},  we construct  elements $f_1, f_1, \ldots, 
f_{s+2}$ in ${\mathcal{K}}[X, Y, Z]^G$ such that
the elements $f_1, \ldots, f_{s+1}$ have the leading monomials
(in fact the leading terms) as  in the Conjecture~\cite{CSW}~(2), and 
the leading monomial of $f_{s+2}$ is same as the leading monomial of $N_G(Z)$. 
 
Let $g_1 = f_{(1, 2, \ldots, r, r+1)}$ and 
$g_2 = f_{(1,2, \ldots, r, r+2)}$ be $(r+1)\times (r+1)$ minors of 
${\tilde \Gamma}$ as given in (\ref{g1e}) and (\ref{g2e}). Then 

$$\begin{array}{lcl}
g_1 & = & X^{2p^{s-1}}\left[a_0(-\frac{\Delta}{X^2})^{p^{s-1}}
-a_1(\frac{Y}{X})^{p^{s-1}}
+a_2(-\frac{\Delta}{X^2})^{p^{s-2}}\ldots -a_r(\frac{Y}{X})\right]\\\
&  = & -a_0\Delta^{p^{s-1}}- a_1X^{p^{s-1}}Y^{p^{s-1}}-
a_2\Delta^{p^{s-2}}X^{2p^{s-1}-2p^{s-2}}\ldots -a_rYX^{2p^{s-1}-1}\end{array}$$
 and
$$\begin{array}{lcl}
g_2 & = & X^{p^s}\left[a_0(\frac{Y}{X})^{p^s} -b_1(\frac{Y}{X})^{p^{s-1}}+
b_2(-\frac{\Delta}{X^2})^{p^{s-2}}
\ldots +b_{r-1}(-\frac{\Delta}{X^2})-b_r(\frac{Y}{X})\right]\\\
& = &  a_0Y^{p^s}- b_1Y^{p^{s-1}}X^{p^s-p^{s-1}} -
b_2\Delta^{p^{s-2}}X^{p^s-2p^{s-2}}\ldots -b_{r-1}\Delta 
X^{p^s-2}-b_rYX^{p^s-1},\end{array}$$
where $a_i, b_i\in {\mathcal{K}}\setminus \{0\}$.
In  Notation~\ref{n2} we give precise description 
of the $a_i$ and the $b_i$.

Let $f_1 = -g_1/a_0$. Then 
\begin{equation}\label{f1} f_1 = \Delta^{p^{s-1}}+ 
\tfrac{a_1}{a_0}X^{p^{s-1}}Y^{p^{s-1}}+
\tfrac{a_2}{a_0}\Delta^{p^{s-2}}X^{2p^{s-1}-2p^{s-2}}\ldots + 
\tfrac{a_r}{a_0}YX^{2p^{s-1}-1}\end{equation}
and let $f_2 = g_2/a_0$ then 
\begin{equation}\label{f2} f_2 = Y^{p^s}- \tfrac{b_1}{a_0}Y^{p^{s-1}}X^{p^s-p^{s-1}} 
-\tfrac{b_2}{a_0}\Delta^{p^{s-2}}X^{p^s-2p^{s-2}}\ldots -\tfrac{b_{r-1}}{a_0}\Delta 
X^{p^s-2}-\tfrac{b_r}{a_0}YX^{p^s-1}.\end{equation}

\begin{notation}\label{n2}Here  $r= 2s-1$.
The elements $a_i$, $b_i$ occurring in (\ref{f1}) and 
(\ref{f2}) are  $r\times r$ minors of the  
matrix $\Gamma$ and are given as follows: A subsequence 
$(j_1, \ldots,  j_{r})$ of $(1, 2, r, r+1, r+2)$, 
determines a $r\times r$ minor, $\nu_{j_1, \ldots,  j_{r}}$,
 of the $(2r+1)\times r$ matrix $\Gamma$. For $0\leq k\leq r$
$$a_0 = b_0 = \nu_{1,2,\ldots,  r, {\hat{r+1}}}\quad
a_{k} = \nu_{1,2,\ldots, {\hat {r+1-k}},\ldots, r, r+1}$$
$$b_0 = a_0\quad \mbox{and}\quad 
  b_{k} = \nu_{1,2,\ldots, {\hat {r+1-k}},\ldots,  r,r+2}.$$
Also $a_i = b_i = 0$ for $i>r$.

Further, let  $\{l_1, \ldots, l_{2m+2n+1}\}$ be a set of $2m+2n+1$ 
integers.
For a sequence $(j_1, \ldots, j_n)$ we define 
$(j_1, \ldots, \{\hat{j_m}, \ldots, j_n\})$ to be the subsequence constructed 
by omitting $j_{m+2l}$ for $l=0, 1, \ldots$, and we denote
$${ \nu}_{l_1, l_2, \ldots, 
\{{\hat{l_{2m}}},  \ldots, l_{2m+2n+1}\}} = 
\nu_{l_1, l_2, \ldots, l_{2m-1}, {\hat {l_{2m}}},
l_{2m+1}, {\hat {l_{2m+2}}}, l_{2m+3},{\hat {l_{2m+4}}},
\ldots, l_{2m+2n+1}}.$$
\begin{enumerate}
\item 
For $-1\leq j\leq s-2 $ and for $0\leq k < r-2j-1$, we define a set of $r\times r$-minors of $\Gamma$ 
as follows: $B(-1, k) = a_k = \nu_{1,2, \ldots,  
{\hat{r+1-k}}, \ldots, r+1}$
and
$$ B(-1) = B(-1,0) = a_0 =  \nu_{1,2, \ldots, r}\neq 0.$$
For $j\geq 0$, let ${B(j, k)} = {\nu}_{1,2, \ldots,  
{{r-2j-1-k}}, \ldots,\{{\hat{r-2j+1}}, 
\ldots, r+2j+2\}} \neq 0$ and
 $$B(j) = B(j, 0) = { \nu}_{1,2, \ldots,  
\{{\hat{r-2j-1}}, \ldots, r+2j+2\}}\neq 0.$$
In particular, for $k\geq 0$ we have  $B(0, k) = b_{k+2}$.
We define  $B(j, k) = 0$ if  $k>r-2j-2$. 
\item For $0\leq j\leq s-2$, the elements
$N_{1j}$, $L_{2j}$ and $N_{2j}$ are the elements of 
${\mathcal{K}}$ given as
$$N_{1j}
 =  \tfrac{B(j-1, 1)}{B(j-1)} -\tfrac{B(j,1)^p}{B(j)^p},\quad\quad
L_{2j} = \tfrac{B(j-1,2)}{B(j-1)} -\tfrac{B(j,2)^p}{B(j)^p},$$
  
$$N_{2j} = \tfrac{B(j-1,3)}{B(j-1)}- L_{2j}\cdot \tfrac{B(j,1)}{B(j)} - 
\tfrac{B(j,3)^p}{B(j)^p}.$$

\item We consider the following  set of elements in the polynomial ring
${\mathcal{K}}[t_1, t_2, \ldots, t_{s+2}]$ where $E_{1j}({\bf \underline{t}})\in 
{\mathcal{K}}[t_1, \ldots, t_{j+2}]$ and 
$E_{2j}({\bf\underline{ t}})\in 
{\mathcal{K}}[t_1, \ldots, t_{j+3}]$. 
Here ${\bf \underline{t}} = (t_1, \ldots, t_{s+2})$. 

\begin{enumerate}
\item  $E_{1,0}({\bf\underline{ t}}) =  t_1^{\frac{p^2+1}{2}}$.\\
\item For $1\leq j\leq s-2$
$$E_{1j}({\bf\underline{ t}}) = 
 t_2^{p^{j+1}-\frac{p^{j}}{2}-p^{j-1}-\frac{p}{2}}(t_1\cdots 
t_{j+2})^{\frac{(p-1)}{2}}t_{j+2}.$$
 
\item For $0\leq j\leq s-2$,  
$$E_{2j}({\bf \underline{t}}) = t_2^{\frac{p^{j+1}}{2}-p^{j}-
\frac{p}{2}}(t_1\cdots 
t_{j+3})^{\frac{(p-1)}{2}}t_{j+3}.$$
\end{enumerate}
\item We consider the following set of elements $F_{j+3}({\bf\underline{ t}})$ 
in the polynomial ring
${\mathcal{K}}[t_1, t_2, \ldots, t_{s+2}][X]$ where   $F_{j+3}({\bf \underline{t}})\in 
{\mathcal{K}}[t_1, \ldots, t_{j+2}][X]$:

\begin{equation}\label{F3t}F_3({\bf\underline{ t}}) = t_2^2- t_1^p+(2b_1/a_0)\cdot 
t_1^{(p+1)/2}X^{p^s-p^{s-1}}.\end{equation}

\begin{multline}\label{F4t}
F_{4}({\bf\underline{ t}}) =  t_{3}^p
-t_2^pt_1+ \left(N_{1,0}\right)\cdot 
E_{1,0}({\bf \underline{t}})X^{p^{s-1}}\\
 + \left(L_{2,0}\right)\cdot  
t_2^{(p-1)}t_{3}X^{2p^{s-1}-2p^{s-2}}
+ \left(N_{2,0}\right)\cdot E_{2,0}({\bf\underline{ t}})X^{2p^{s-1}-p^{s-2}}.
\end{multline}

For $1\leq j\leq s-2$
\begin{multline}\label{Fj+4t}
F_{j+4}({\bf\underline{ t}}) =  t_{j+3}^p
-t_2^{p^{j-1}(p^2-1)}t_{j+2}+ \left(N_{1j}\right)\cdot 
E_{1j}({\bf \underline{t}})X^{p^{s-1-j}}\\\\
 + \left(L_{2j}\right)\cdot  
t_2^{p^{j}(p-1)}t_{j+3}X^{2p^{s-1-j}-2p^{s-2-j}}
+ \left(N_{2j}\right)\cdot E_{2j}({\bf\underline{ t}})X^{2p^{s-1-j}-p^{s-2-j}}.
\end{multline}
\end{enumerate}
\end{notation}

\vspace{5pt}

Now after making choices of $E_{1j}({\bf \underline{t}})$ and 
$E_{2j}({\bf \underline{t}})$ it is easy to check the 
leading terms of polynomials as given in the following lemma.
Note that for $f_1$ and $f_2$ as in (\ref{f1}) and (\ref{f2}) we have
$LT(f_1) = Y^{2p^{s-1}}$ and $LT(f_2) = Y^{p^s}$.

\begin{lemma}\label{LM}Let $r\geq 3$ be an odd integer and let $s = (r+1)/2$. 
Suppose there exists an integer
 $j_1\leq s-2$  and homogeneous polynomials $f_3, \ldots, f_{j_1+3} \in {\mathcal{K}}[X,Y, Z]$ 
such that 
for every $j$, where $0\leq j\leq j_1$,  we have $LT(f_{j+3}) = Y^{p^{s+j}+2p^{s-2-j}}$.

Then 

\begin{enumerate}
\item For $1\leq j\leq j_1 $, we have  
$LT(f_2^{p^{j-1}(p^2-1)}f_{j+2}) = Y^{p^{s+1+j}+2p^{s-1-j}}$.\\

\item For $0\leq j\leq j_1$, we have  
\begin{enumerate}
\item 
$LT(f_2^{p^{j}(p-1)}f_{j+3}) = Y^{p^{s+1+j}+2p^{s-2-j}}$,\\

\item $LT(E_{1j}({\bf\underline{ f}})) = Y^{p^{s+1+j}+p^{s-1-j}}$,\\

\item  $LT(E_{2j}({\bf\underline{ f}})) = Y^{p^{s+1+j}+p^{s-2-j}}.$
\end{enumerate}
\end{enumerate}
 \end{lemma}

\begin{proposition}\label{odd}
Let $r=2s-1\geq 3$ be an integer and $p>2$ be a prime number.
Let $f_1, f_2$ be the elements as in (\ref{f1}) and (\ref{f2}). We recursively define the elements 
$f_3,\ldots, f_{s+1}, f_{s+2}$ in ${\mathcal{K}}[X, Y, Z][X^{-1}]$ 
as follows:

\begin{align*}
& f_3  = \left(\tfrac{a_0}{2b_2}\right)\frac{f_1^p-f_2^2}{X^{p^s-2p^{s-2}}}-
\left(\tfrac{b_1}{b_2}\right)\cdot\frac{ f_1^{(p+1)/2}}{X^{p^{s-1}-2p^{s-2}}},\\\\
& f_4  = \tfrac{B(0)^{p+1}}{B(-1)^p B(1,0)}\cdot  \frac{F_4({\bf\underline{ f}})}{ 
X^{2p^{s-1}-2p^{s-3}}},\\\\
& f_{j+4}  = \tfrac{B(j)^{p+1}}{(B(j-1)^pB(j+1)}
\frac{F_{j+4}({\bf\underline{ f}})}{X^{2p^{s-1-j}-2p^{s-3-j}}},\quad\mbox{for}\quad
 0\leq j\leq s-3, \\\\
& f_{s+2}  = \tfrac{2^{p^{s-1}}B(s-2)^p}{B(s-3)^p} 
\cdot\frac{F_{s+2}({\bf\underline{ f}})}{X^{2p}}.
\end{align*}

Then  
\begin{enumerate}  
\item[(1)]
$f_1, \ldots, f_{s+1}, f_{s+2}$ are  homogeneous
polynomials and belong to  the 
ring ${\mathcal{K}}[X, Y, Z]^G$, 
\item[(2)] $ LT(f_{s+2}) = -Z^{p^r}$ and
$$LT(f_1) = Y^{2p^{s-1}}, \quad LT(f_2) = Y^{p^s}\quad 
 LT(f_{j+3}) = Y^{p^{s+j}+2p^{s-2-j}},$$
for $0\leq j \leq s-2$.
\end{enumerate}
\end{proposition}

\begin{rmk} Lemma~\ref{MLo} implies that    
$f_3, \ldots, f_{s+2}$ which are given as above are indeed well 
defined.
\end{rmk}

\begin{proof}We note that the elements $f_1$ and $f_2$ satisfy 
the conditions (1) and (2) of the above proposition.
Now we show that the element $f_3$ satisfies these conditions.

By (\ref{f1})  (recall
$B(-1, k) = a_k$ and $B(-1) = a_0$ and  $\Delta = Y^2-XZ$),

$$f_1  =  
\sum_{l=2}^{s+1}\left[\tfrac{B(-1, 2l-4)}{B(-1)}
\Delta^{p^{s-l+1}}X^{2p^{s-1}-2p^{s-l+1}}+
\tfrac{B(-1, 2l-3)}{B(-1)} Y^{p^{s-l+1}}X^{2p^{s-1}-p^{s-l+1}}\right]$$

and  (recall  $B(0, k) = b_{k+2}$ for $k+2 \leq r$)

\begin{multline*}
f_2 = Y^{p^s}- \tfrac{b_1}{a_0}Y^{p^{s-1}}X^{p^s-p^{s-1}} 
-
\sum_{l=2}^s\left[\tfrac{B(0, 2l-4)}{B(-1)}\Delta^{p^{s-l}}X^{p^s-2p^{s-l}}+
\tfrac{B(0, 2l-3)}{B(-1)}Y^{p^{s-l}}X^{p^s-p^{s-l}}\right].\end{multline*}

Let $I = X^{p^s}(X,Y)$ denote the ideal in ${\mathcal{K}}[X, Y, Z]$. 
Then 
$f_1^p\equiv_{I} Y^{2p^s}-X^{p^s}Z^{p^s}$ and 

\begin{multline*}f_2^2\equiv_{I} Y^{2p^s}
- \left(\tfrac{2b_1}{a_0}\right)Y^{p^{s}+p^{s-1}}X^{p^s-p^{s-1}}\\\\ 
-2Y^{p^s}
\sum_{l=2}^s\left[\tfrac{B(0, 2l-4)}{B(-1)}\Delta^{p^{s-l}}X^{p^s-2p^{s-l}}+
\tfrac{B(0, 2l-3)}{B(-1)}Y^{p^{s-l}}X^{p^s-p^{s-l}}\right].\end{multline*}

Also 
$$X^{p^s-p^{s-1}}f_1^{\frac{p+1}{2}}\equiv_{I}  X^{p^s-p^{s-1}}(\Delta^{p^{s-1}})^{\frac{p+1}{2}}
\implies 
X^{p^s-p^{s-1}}f_1^{\frac{p+1}{2}}\equiv_{I} 
Y^{p^s+p^{s-1}}X^{p^s-p^{s-1}}.$$

Hence
\begin{multline*}f_2^2-f_1^p +\tfrac{2b_1}{a_0}\cdot f_1^{(p+1)/2}X^{p^s-p^{s-1}}
\equiv_{I} X^{p^{s}}Z^{p^s}\\\\
-
2Y^{p^s}
\sum_{l=2}^s\left[\tfrac{B(0, 2l-4)}{B(-1)}\Delta^{p^{s-l}}X^{p^s-2p^{s-l}}+
\tfrac{B(0, 2l-3)}{B(-1)}Y^{p^{s-l}}X^{p^s-p^{s-l}}\right].\end{multline*}

This implies that the left hand side is divisible by $X^{p^s-2p^{s-2}}$ 
and therefore

\begin{equation}\label{f3}
f_3= -\left(\tfrac{a_0}{2b_2}\right)\cdot \frac{f_2^2- f_1^p+(2b_1/a_0)\cdot 
f_1^{(p+1)/2}X^{p^s-p^{s-1}}}{X^{p^s-2p^{s-2}}}\end{equation}
is an homogeneous element satisfying condition (1).

Further  for the ideal  $I_4 = X^{2p^{s-2}}(X, Y)$ 
we have 

\begin{multline}f_3  \equiv_{I_4}\\\ -
\tfrac{B(-1)}{2B(0)}X^{2p^{s-2}}Z^{p^s}+
Y^{p^s}
\sum_{l=2}^s\left[\tfrac{B(0, 2l-4)}{B(0)}\Delta^{p^{s-l}}X^{2p^{s-2}-2p^{s-l}}+
\tfrac{B(0, 2l-3)}{B(0)}Y^{p^{s-l}}X^{2p^{s-2}-p^{s-l}}\right]\end{multline}

which implies 
$LT(f_3) = Y^{p^s+2p^{s-2}}$.

\vspace{5pt}

\noindent{\bf Claim}.\quad 
 Let  $I_{j+3} = 
X^{2p^{s-1-j}}(X,Y)$ denote the ideal in ${\mathcal{K}}[X,Y,Z]$ and 
let  $0\leq j\leq s-3$.
\begin{enumerate}
 \item[(1)] If $f_3,\ldots, f_{j+3} \in {\mathcal{K}}[X,Y,Z]^G$ are homogeneous polynomials of degree 
 $p^{s+j}+2p^{s-2-j}$ and
 \item[(2)] the element $f_{j+3}$ mod the ideal $I_{j+4}$   has the  expression

\begin{multline}\label{f_{j+3}} f_{j+3}   \equiv_{I_{j+4}}  
-\left(\tfrac{B(j-1)}{2^{p^j}B(j)}\right)X^{2p^{s-2-j}}Z^{p^{s+j}}\\\
+ Y^{p^{s+j}}\sum_{l=2}^{s-j}\left[\tfrac{B(j,2l-4)}{B(j)}
\Delta^{p^{s-l-j}}X^{2p^{s-2-j}-2p^{s-l-j}}
+\tfrac{B(j,2l-3)}{B(j)}Y^{p^{s-l-j}}X^{2p^{s-2-j}-p^{s-l-j}}\right]
\end{multline}
then  the element $f_{j+4}$ satisfies  conditions (1) and (2).
\end{enumerate}

\noindent{\underline{Proof of the claim}}:\quad 
Note that the condition(2) implies that $LT(f_{j+3}) = Y^{p^{s+j}+2p^{s-2-j}}$. 
We will prove the claim along the following lines. Let $F_{j+4}({\underline {\bf f}})$, denote 
the evaluation of $F_{j+4}({\underline {\bf t}})$ at 
${\underline {\bf t}} = {\underline {\bf f}}$, where ${\underline {\bf f}} = (f_1, \ldots, f_{j+3})$. 
Then it 
is a homogeneous element in ${\mathcal{K}}[X,Y,Z]^G$ and of degree $2p^{s-1-j}+p^{s+1+j}$. We will show that 
$F_{j+4}({\underline {\bf f}})$ mod the ideal $I_{j+3}$ has the expression as in (\ref{Fj+44}). In particular $F_{j+4}({\underline {\bf f}})$ is divisible by 
$X^{2p^{s-1-j}-2p^{s-3-j}}$ in ${\mathcal{K}}[X, Y, Z]$.
Since $(X^{2p^{s-1-j}-2p^{s-3-j}})I_{j+5} = I_{j+3}$
we get the expression like (\ref{f_{j+3}}) for the  element $f_{j+4}$ mod the ideal $I_{j+5}$, where it is obvious that $f_{j+4}$ satisfies the condition (1) of the claim.
 
Note that
\begin{equation}\label{ideal}
(I_{j+4})^p \subseteq I_{j+3}\subseteq I_{j+4},\;\: \mbox{for}\;\; 
0\leq j\leq s-3.\end{equation}

\vspace{5pt}

We will make use of the following set of equalities, where Eq.(3), Eq.(4) and 
 Eq.(5)   can be proved  using  (\ref{f_{j+3}}) and the induction hypothesis.

\begin{enumerate}
\item[Eq.(1)] For every $j\geq 0$,
 we have  $f_2\equiv_{I_{j+3}}Y^{p^s}$.\\

\item[Eq.(2)] $f_2^pf_1 \equiv_{I_{j+3}} Y^{p^{s+1}}f_1$.\\

\item[Eq.(3)] For $1\leq j\leq s-2$,  
$f_2^{p^{j-1}(p^2-1)}f_{j+2}\equiv_{I_{j+3}} 
Y^{p^{s+j+1}-p^{s+j-1}}f_{j+2}.$\\
\item[Eq.(4)] For $0\leq j\leq s-2$,
$ f_2^{p^{j}(p-1)}f_{j+3}  \equiv_{I_{j+4}}
 Y^{p^{s+j+1}-p^{s+j}}f_{j+3}$, 

 in fact
$ f_2^{p^{j}(p-1)}f_{j+3}  \equiv_{I_{j+3}}
 Y^{p^{s+j+1}-p^{s+j}}f_{j+3}$.\\
\item[Eq.(5)]For  $0\leq j\leq s-2$,
\begin{enumerate}
\item[(i)] $X^{p^{s-1-j}}E_{1j}({\bf\underline{ f}}) \equiv_{I_{j+3}} 
X^{p^{s-1-j}}Y^{p^{s+1+j}+p^{s-1-j}},$
and 
\item[(ii)] 
$X^{2p^{s-1-j}-p^{s-2-j}}E_{2j}({\bf\underline{ f}}) \equiv_{I_{j+3}} 
X^{2p^{s-1-j}-p^{s-2-j}}Y^{p^{s+1+j}+p^{s-2-j}}.$
\end{enumerate}
\end{enumerate}

We prove the claim by induction on $j$, where $0\leq j\leq s-3$.

The claim holds for $j=0$.
We   assume that
the expression~(\ref{f_{j+3}}) holds for 
 $f_3, \ldots, f_{j+3}$. In particular 
$LT(f_{j_1+3}) = Y^{p^{s+j_1}+2p^{s-2-j_1}}$, for $0\leq j_1\leq j$.

 Evaluating $F_4({\bf\underline{ t}})$ at ${\bf\underline{t}} = 
{\bf\underline{ f}}$, where ${\bf\underline{f}} = (f_1, f_2, f_3)$ we get 

$$\begin{array}{lcl}
F_4({\bf\underline{ f}}) & = & f_3^p-  f_1f_2^p+ \left(N_{1,0}\right)\cdot E_{1,0}(\underline{f})X^{p^{s-1}}+
\left(L_{2,0}\right)\cdot f_2^{p-1}f_3
X^{2p^{s-1}-2p^{s-2}}\\\\
& & + \left(N_{2,0}\right) \cdot E_{2,0}(\underline{f})
X^{2p^{s-1}-p^{s-2}},\end{array}$$

and 
evaluating $F_{j+4}({\bf\underline{ t}})$ at ${\bf\underline{t}} = 
{\bf\underline{ f}}$, where ${\bf\underline{f}} = (f_1, f_2, \ldots, f_{j+3})$  and $j\geq 1$ we get 

$$\begin{array}{lcl}
F_{j+4}({\bf \underline{f}}) & = & f_{j+3}^p
-f_2^{p^{j-1}(p^2-1)}f_{j+2}+ \left(N_{1j}\right)\cdot 
E_{1j}({\bf\underline{ f}})X^{p^{s-1-j}}\\\\
& & + \left(L_{2j}\right)\cdot  
f_2^{p^{j}(p-1)}f_{j+3}X^{2p^{s-1-j}-2p^{s-2-j}}
+ \left(N_{2j}\right)\cdot E_{2j}({\bf\underline{ f}})X^{2p^{s-1-j}-p^{s-2-j}}.
\end{array}$$

Now to prove the claim 
it is enough to prove the following equality

\begin{multline}\label{Fj+44}
F_{j+4}({\bf\underline{ f}})  \equiv_{I_{j+3}}  
-\left(\tfrac{B(j-1)}{2^{p^j}B(j)}\right)^p\cdot
X^{2p^{s-1-j}}Z^{p^{s+1+j}}\\\\
+ \tfrac{B(j-1)^p}{B(j)^{p+1}}
 Y^{p^{s+j+1}}\sum_{l=2}^{s-(j+1)}\left[B(j+1,2l-4)
\Delta^{p^{s-l-(j+1)}}X^{2p^{s-(j+1)}-2p^{s-l-(j+1)}}\right.\\\\
\left.+ B(j+1,2l-3)
Y^{p^{s-l-(j+1)}}X^{2p^{s-(j+1)}-p^{s-l-(j+1)}}\right].
\end{multline}

We note that the second and third terms of 
$F_{j+4}({\underline{\bf{f}}})$  at $j=0$,
differ from the second and third  terms of $F_{4}({\underline{\bf{f}}})$. Hence
we consider the cases $j=0$ and $j\geq 1$ separately,
for the sum  of the first three terms.

\vspace{5pt}

\noindent{\underline{Case}}. Let $j=0$.

By definition  $I_3 = 
X^{2p^{s-1}}(X, Y)$. It is easy to check

\begin{multline*}f^p_3 \equiv_{I_3} -\left(\tfrac{B(-1)}{2B(0)}\right)^pX^{2p^{s-1}}
Z^{p^{s+1}}
+ Y^{p^{s+1}}\sum_{l=2}^{s}\left[\left(\tfrac{B(0,2l-4)^p}{B(0)^p}\right)
\Delta^{p^{s-l+1}}X^{2p^{s-1}-2p^{s-l+1}}\right.\\\\
\left.+\left(\tfrac{B(0,2l-3)^p}{B(0)^p}\right) Y^{p^{s-l+1}}X^{2p^{s-1}-p^{s-l+1}}
\right].
\end{multline*}

Therefore, by Eq.(2)

$$f_1f_2^p   \equiv_{I_3} Y^{p^{s+1}}
\sum_{l=2}^{s+1}\left[\tfrac{B(-1, 2l-4)}{B(-1)}
\Delta^{p^{s-l+1}}X^{2p^{s-1}-2p^{s-l+1}}+
\tfrac{B(-1, 2l-3)}{B(-1)} Y^{p^{s-l+1}}X^{2p^{s-1}-p^{s-l+1}}\right].$$

By definition  $N_{1,0}$
is given by 
\begin{equation}\label{N_0}
N_{1,0} +\tfrac{B(0,1)^p}{B(0)^p}-\tfrac{B(-1,1)}{B(-1)}
= 0\end{equation}
and therefore is in the field ${\mathcal{K}}$.
Then by Eq.(5)

\begin{multline}\label{u4}f_3^p-f_2^pf_1 + N_{1,0}\cdot 
E_{1,0}(\underline{f})
X^{p^{s-1}} \equiv_{I_3} -
\left(\tfrac{B(-1)}{2B(0)}\right)^p\cdot 
X^{2p^{s-1}}Z^{p^{s+1}}\\\\
+ Y^{p^{s+1}}\sum_{l=3}^{s}\left[\left(\tfrac{B(0,2l-4)^p}{B(0)^p}- \tfrac{B(-1, 2l-4)}{B(-1)}\right)
\Delta^{p^{s-l+1}}X^{2p^{s-1}-2p^{s-l+1}}\right.\\\\
\left.+\left(\tfrac{B(0,2l-3)^p}{B(0)^p}-
\tfrac{B(-1, 2l-3)}{B(-1)} Y^{p^{s-l+1}}X^{2p^{s-1}-p^{s-l+1}}\right)
\right]\\\\
+Y^{p^{s+1}}\left[-\tfrac{B(-1,r-1)}{B(-1)}
\Delta X^{2p^{s-1}-2}-\tfrac{B(-1,r)}{B(-1)}
YX^{2p^{s-1}-1}\right].\end{multline}

\vspace{5pt}

\noindent{\underline{Case}}.\quad Let $1\leq j\leq s-3$.

Since $(I_{j+4})^p \subseteq I_{j+3}$ we get

\begin{multline}\label{fj+3} f^p_{j+3}   \equiv_{I_{j+3}}  
-\left(\tfrac{B(j-1)^p}{2^{p^{j+1}}B(j)^p}\right)
X^{2p^{s-1-j}}Z^{p^{s+1+j}}\\\\
+ Y^{p^{s+j+1}}\sum_{l=2}^{s-j} \left[
\tfrac{B(j,2l-4)^p}{B(j)^p}
\Delta^{p^{s-l-j+1}}X^{2p^{s-1-j}-2p^{s-l-j+1}}
+\tfrac{B(j,2l-3)^p}{B(j)^p}
Y^{p^{s-l-j+1}}X^{2p^{s-1-j}-p^{s-l-j+1}}\right].
\end{multline}

By Eq.(3) we have 
\begin{multline}\label{f_2}
 f_2^{p^{j-1}(p^2-1)}f_{j+2}  \equiv_{I_{j+3}}  
Y^{p^{s+j+1}}\sum_{l=2}^{s-(j-1)}\left[\tfrac{B(j-1,2l-4)}{B(j-1)}
\Delta^{p^{s-l-(j-1)}}X^{2p^{s-2-(j-1)}-2p^{s-l-(j-1)}}\right.\\\\
\left.+\tfrac{B(j-1,2l-3)}{B(j-1)}Y^{p^{s-l-(j-1)}}X^{2p^{s-2-(j-1)}-p^{s-l-(j-1)}}\right].
\end{multline}

By (\ref{fj+3}) and (\ref{f_2}) we get

\begin{multline*} f^p_{j+3} -  f_2^{p^{j-1}(p^2-1)}f_{j+2}  \equiv_{I_{j+3}}  
-\left(\tfrac{B(j-1)^p}{2^{p^{j+1}}B(j)^p}\right)
X^{2p^{s-1-j}}Z^{p^{s+1+j}}\\\\
+ Y^{p^{s+j+1}}\left[\sum_{l=3}^{s-j} \left[\tfrac{B(j,2l-4)^p}{B(j)^p}-   
\tfrac{B(j-1,2l-4)}{B(j-1)}\right]
\Delta^{p^{s-l-j+1}}X^{2p^{s-1-j}-2p^{s-l-j+1}}\right.\\\\
\left.+\sum_{l=2}^{s-j} 
\left[\tfrac{B(j,2l-3)^p}{B(j)^p}-\tfrac{B(j-1, 2l-3}{B(j-1)}\right]
Y^{p^{s-l-j+1}}X^{2p^{s-1-j}-p^{s-l-j+1}}\right]\\\\
+Y^{p^{s+j+1}}\left[-\tfrac{B(j-1,r-2j-1))}{B(j-1)}
\Delta X^{2p^{s-1-j}-2} 
-\tfrac{B(j-1,r-2j)}{B(j-1)}Y X^{2p^{s-1-j}-1}\right].
\end{multline*}

By Eq.(5)(i),  for $1\leq j\leq s-3$, we have 
$X^{p^{s-1-j}}E_{1j}({\bf\underline{ f}}) \equiv_{I_{j+3}} 
X^{p^{s-1-j}}Y^{p^{s+1+j}+p^{s-1-j}}$.

Further we have
 $N_{1j}\in {\mathcal{K}}$
such that
\begin{equation}\label{N_1}
N_{1j} +\tfrac{B(j,1)^p}{B(j)^p}-\tfrac{B(j-1, 1)}{B(j-1)} = 0.\end{equation}

This gives
\begin{multline}\label{uj+4} f^p_{j+3} -  f_2^{p^{j-1}(p^2-1)}f_{j+2}  
+
(N_{1j})\cdot
E_{1j}({\bf\underline{ f}})X^{p^{s-1-j}}
\equiv_{I_{j+3}}  
-\left(\tfrac{B(j-1)^p}{2^{p^{j+1}}B(j)^p}\right)
X^{2p^{s-1-j}}Z^{p^{s+1+j}}\\\\
+ Y^{p^{s+j+1}}\left[\sum_{l=3}^{s-j} \left[\tfrac{B(j,2l-4)^p}{B(j)^p}-   
\tfrac{B(j-1,2l-4)}{B(j-1)}\right]
\Delta^{p^{s-l-j+1}}X^{2p^{s-1-j}-2p^{s-l-j+1}}\right.\\\\
\left.+\sum_{l=3}^{s-j} 
\left[\tfrac{B(j,2l-3)^p}{B(j)^p}-\tfrac{B(j-1, 2l-3}{B(j-1)}\right]
Y^{p^{s-l-j+1}}X^{2p^{s-1-j}-p^{s-l-j+1}}\right]\\\\
+Y^{p^{s+j+1}}\left[-\tfrac{B(j-1,r-2j-1))}{B(j-1)}
\Delta X^{2p^{s-1-j}-2} 
-\tfrac{B(j-1,r-2j)}{B(j-1)}Y X^{2p^{s-1-j}-1}\right].
\end{multline}

By (\ref{u4}) and (\ref{uj+4}) we have a uniform expression
for the first three terms of  
$F_{j+4}(\bf\underline{f})$. Henceforth we assume  $0\leq j\leq s-3$.

Now, by Eq.(4) we have  $f_2^{p^{j}(p-1)}f_{j+3}  \equiv_{I_{j+4}}
 Y^{p^{s+j+1}-p^{s+j}}f_{j+3}$ and, by definition 
$X^{2p^{s-1-j}-2p^{s-2-j}}I_{j+4} = I_{j+3}$. 

Therefore 
$$f_2^{p^{j}(p-1)}f_{j+3}X^{2p^{s-1-j}-2p^{s-2-j}}  \equiv_{I_{j+3}}
 Y^{p^{s+j+1}-p^{s+j}}f_{j+3}X^{2p^{s-1-j}-2p^{s-2-j}}.$$ 

Hence by (\ref{f_{j+3}}) and Eq.~(4) we get

\begin{multline*}
f_2^{p^{j}(p-1)}f_{j+3}X^{2p^{s-1-j}-2p^{s-2-j}}  \equiv_{I_{j+3}}\\\\
Y^{p^{s+j+1}}\sum_{l=2}^{s-j}
\left[\tfrac{B(j,2l-4)}{B(j)}
\Delta^{p^{s-l-j}}X^{2p^{s-1-j}-2p^{s-l-j}}
 +\tfrac{B(j,2l-3)}{B(j)}
Y^{p^{s-l-j}}X^{2p^{s-1-j}-p^{s-1-j}}\right].\end{multline*}

We recall  that $L_{2j}\in {\mathcal{K}}$ such that 
\begin{equation}\label{L_2}L_{2j} + \tfrac{B(j,2)^p}{B(j)^p}-\tfrac{B(j-1,2)}{B(j-1)}
 = 0.\end{equation}
Then 
\begin{multline*}f_{j+3}^p- f_2^{p^{j-1}(p^2-1)}f_{j+2} +(N_{1j})\cdot
E_{1j}({\bf\underline{ f}})X^{p^{s-1-j}} + \left(L_{2j}\right)\cdot
f_2^{p^{j}(p-1)}f_{j+3}X^{2p^{s-1-j}-2p^{s-2-j}}\\\\
\equiv_{I_{j+3}}
-\left(\tfrac{B(j-1)^p}{2^{p^{j+1}}B(j)^p}\right)
X^{2p^{s-1-j}}Z^{p^{s+1+j}}\\\\
+ Y^{p^{s+j+1}}\left(\sum_{l=4}^{s-j} \left[
L_{2j}\tfrac{B(j,2l-6)}{B(j)}+
\tfrac{B(j,2l-4)^p}{B(j)^p}-   
\tfrac{B(j-1,2l-4)}{B(j-1)}\right]
\Delta^{p^{s-l-j+1}}X^{2p^{s-1-j}-2p^{s-l-j+1}}\right.\\\\
\left.+\sum_{l=3}^{s-j} 
\left[L_{2j}\tfrac{B(j,2l-5)}{B(j)}+
\tfrac{B(j,2l-3)^p}{B(j)^p}-\tfrac{B(j-1, 2l-3}{B(j-1)}\right]
Y^{p^{s-l-j+1}}X^{2p^{s-1-j}-p^{s-l-j+1}}\right)\\\\
+Y^{p^{s+j+1}}\left[
L_{2j}\tfrac{B(j,r-2j-3)}{B(j)}
-\tfrac{B(j-1,r-2j-1))}{B(j-1)}\right]
\Delta X^{2p^{s-1-j}-2}\\\\
+Y^{p^{s+j+1}}\left[
L_{2j}\tfrac{B(j,r-2j-2)}{B(j)}+
-\tfrac{B(j-1,r-2j)}{B(j-1)}\right]Y X^{2p^{s-1-j}-1}.\end{multline*}

By Eq.(5)(ii),  we have 
$X^{2p^{s-1-j}-p^{s-2-j}}E_{2j}({\bf\underline{ f}}) \equiv_{I_{j+3}} 
X^{2p^{s-1-j}-p^{s-2-j}}Y^{p^{s+1+j}+p^{s-2-j}}$, for $0\leq j \leq s-3$.
By definition for $N_{2j}\in {\mathcal{K}}$ 
 such that 
\begin{equation}\label{N_2}
N_{2j}+L_{2j}\cdot \tfrac{B(j,1)}{B(j)} + 
\tfrac{B(j,3)^p}{B(j)^p}-\tfrac{B(j-1,3)}{B(j-1)}= 0.\end{equation}

This gives

\begin{multline}
 F_{j+4}({\bf \underline{f}})  
\equiv_{I_{j+3}} -\left(\tfrac{B(j-1)^p}{2^{p^{j+1}}B(j)^p}\right)
X^{2p^{s-1-j}}Z^{p^{s+1+j}} \\\\
 + Y^{p^{s+1+j}}\left[\sum_{l=4}^{s-j}\left(
L_{2j}\cdot \tfrac{B(j,2l-6)}{B(j)} +\tfrac{B(j,2l-4)^p}{B(j)^p}-\tfrac{B(j-1,2l-4)}{B(j-1)}\right)
\Delta^{p^{s-l+1-j}}X^{2p^{s-1-j}-2p^{s-l+1-j}}\right. \\\\ 
+\left.\left(L_{2j}\cdot \tfrac{B(j,2l-5)}{B(j)} +\tfrac{B(j,2l-3)^p}{B(j)^p}-\tfrac{B(j-1,2l-3)}{B(j-1)}\right)
Y^{p^{s-l+1-j}}X^{2p^{s-1-j}-p^{s-l+1-j}}\right]\\\\
+ Y^{p^{s+1+j}}
\left[L_{2j}\cdot \tfrac{B(j,r-2j-3)}{B(j)}
-\tfrac{B(j-1,r-2j-1)}{B(j-1)}\right]\Delta X^{2p^{s-1-j}-2}\\\\
+
Y^{p^{s+1+j}}
\left[L_{2j}\cdot \tfrac{B(j,r-2j-2)}{B(j)}- \tfrac{B(j-1,r-2j)}{B(j-1)}\right]YX^{2p^{s-1-j}-1}.\end{multline}

We recall that $B(j ,r-2j-1) = B(j, r-2j) = 0$.
 The claim follows from Lemma~\ref{MLo}.

\vspace{5pt}

Now we have homogeneous polynomials  $f_1, \ldots , f_{s+1}$ 
in ${\mathcal{K}}[X, Y, Z]^G$ with the leading terms as in assertion~(2) of the 
proposition. It only remains to prove the assertion~(1) and (2)  for $f_{s+2}$.

Note that the number of monomials in $f_{j+3}$
modulo ${I_{j+4}}$ is $r-2j$. 
In particular, for $I_{s+2} = X^2(X, Y)$

$$f_{s+1} = f_{(s-2)+3} \equiv_{I_{s+2}} -\tfrac{B(s-3)}{2^{p^{s-2}}B(s-2)}X^2Z^{p^{2s-2}} 
+\Delta Y^{p^{2s-2}} + \tfrac{B(s-2,1)}{B(s-2)}Y^{p^{2s-2}+1} X.$$
Therefore (recall  $I_{s+1} = I_{(s-2)+3} = X^{2p}(X, Y)$)

$$f^p_{s+1} \equiv_{I_{s+1}} -\left(\tfrac{B(s-3)}{2^{p^{s-2}}B(s-2)}\right)^p
X^{2p}Z^{p^{2s-1}} 
+\Delta^p Y^{p^{2s-1}} + \left(\tfrac{B(s-2,1)}{B(s-2)}\right)^p Y^{p^{2s-1}+p} X^p.$$

Consider
$$\begin{array}{lcl}F_{s+2}({\bf \underline{f}}) & = & f_{s+1}^p -
f_2^{p^{s-3}(p^2-1)}f_{s}+  \left(N_{1,{s-2}}\right)\cdot 
E_{1,{s-2}}({\bf \underline{f}})X^{p}\\\
&&+ \left(L_{2,{s-2}}\right)
\cdot  f_2^{p^{s-2}(p-1)}f_{s+1}X^{2p-2}
+ \left(N_{2,{s-2}}\right) \cdot E_{2,{s-2}}({\bf \underline{f}})X^{2p-1}.\end{array}$$

Note, by Eq.(3)

$$\begin{array}{lcl}f_2^{p^{s-3}(p^2-1)}f_s & \equiv_{I_{s+1}} & 
\Delta^{p}Y^{p^{2s-1}}+
\tfrac{B(s-3,1)}{B(s-3)}Y^{p^{2s-1}+p}X^p \\\\
& & + \tfrac{B(s-3,2)}{B(s-3)}\Delta Y^{p^{2s-1}}X^{2p-2} + 
\tfrac{B(s-3,3)}{B(s-3)}
Y^{p^{2s-1}+1}X^{2p-1}.\end{array}$$

Further, by definition,
$N_{1,{s-2}}\in {\mathcal{K}}$
such that
\begin{equation*}
N_{1,{s-2}} +\tfrac{B(s-2,1)^p}{B(s-2)^p}-\tfrac{B(s-3,1)}{B(s-3)}
= 0.\end{equation*}

Then, by Eq.(5)

\begin{multline*}f_{s+1}^p -
f_2^{p^{s-3}(p^2-1)}f_{s}+  \left(N_{1,{s-2}}\right)\cdot 
E_{1,{s-2}}({\bf \underline{f}})X^{p}
\equiv_{I_{s+1}}\\\\ 
-\left(\tfrac{B(s-3)}{2^{p^{s-2}}B(s-2)}\right)^p
X^{2p}Z^{p^{2s-1}}
- \tfrac{B(s-3,2)}{B(s-3)}\Delta Y^{p^{2s-1}}X^{2p-2} -
\tfrac{B(s-3,3)}{B(s-3)}Y^{p^{2s-1}+1}X^{2p-1}.
\end{multline*}

By definition  $L_{2,{s-2}}\in {\mathcal{K}}$ such that 
\begin{equation*}L_{2,{s-2}} + \tfrac{B(s-2,2)^p}{B(s-2)^p}-
\tfrac{B(s-3,2)}{B(s-3)} = 
L_{2,{s-2}} -\tfrac{B(s-3,2)}{B(s-3)} = 
0\end{equation*}

and  $N_{2,{s-2}}\in {\mathcal{K}}$ such that 
\begin{equation*}N_{2,{s-2}}+L_{2,{s-2}}\cdot \tfrac{B(s-2,1)}{B(s-2)} -\tfrac{B(s-3,3)}{B(s-3)}= 0.\end{equation*}

This gives  
$$  F_{s+2}({\bf\underline{f}}) \equiv_{I_{s+1}} -\left(\tfrac{B(s-3)}{2^{p^{s-2}}B(s-2)}\right)^p
X^{2p}Z^{p^{2s-1}}.$$

By definition
\begin{equation}\label{fs+2}
f_{s+2} = \left(\tfrac{2^{p^{s-1}}B(s-2)^p}{B(s-3)^p}\right)
\frac{F_{s+2}({\bf \underline{f}})}{X^{2p}}\quad\mbox{which implies}\quad
f_{s+2}\equiv_{I_{s+3}} -Z^{p^{2s-1}}.
\end{equation}

Hence $LT(f_{s+2}) = -Z^{p^r}$. \end{proof}

Now we prove the lemma which played a crucial role for the  induction process
 in Proposition~\ref{odd}, that is, to construct  $f_{j+4}$ from 
$f_1, \ldots, f_{j+3}$.

Moreover it gives an explicit 
 formula  of $f_{j+4}$, mod the ideal  $X^{2p^{s-3-j}}(X, Y)$,
in terms of the minors $B(j+1,k)$, $B(j)$ and $B(j-1)$, which might be of use 
to compute the invariants of all three dimensional representations of $(\Z/p\Z)^r$.

\begin{lemma}\label{MLo} 
If  $0\leq j\leq s-2$ and 
$0\leq k \leq r-2(j+1)-2$, then  
$$A^{(j+1)}_{k} := L_{2j}\cdot \tfrac{B(j,2+k)}{B(j)} + 
\tfrac{B(j,4+k)^p}{B(j)^p}-\tfrac{B(j-1,4+k)}{B(j-1)} 
= \tfrac{B(j-1)^p B(j+1,k)}{B(j)^{p+1}},$$
where 
$L_{2j} = \tfrac{B(j-1,2)}{B(j-1)} - \tfrac{B(j,2)^p}{B(j)^p}$ and  
$B(j, r-2j-1) = B(j, r-2j) = 0$.

In particular $A_k^{(j+1)}\neq 0$, if $0\leq k \leq r-2j-4$.
\end{lemma} 
\begin{proof}We can rewrite 
\begin{multline*}
A^{(j+1)}_{k} = \tfrac{1}{B(j)^{p+1}B(j-1)}\Bigl[B(j-1)\Bigl(B(j)B(j, 4+k)^p-
B(j,2)^pB(j, 2+k)\Bigr)\Bigr. \\
+ \Bigl.B(j)^p\Bigl(B(j-1,2)B(j, 2+k)-
B(j)B(j-1, 4+k)\Bigr)\Bigr].\end{multline*}

Following Notation~\ref{n2}, we have $B(0, k) = b_{k+2}$ and $B(-1, k) = a_k$.

\vspace{5pt}

\noindent{\underline{Case}~(1)}.\quad Let $j=0$ and $0\leq k\leq r-4$. Then 

$$A^{(1)}_k = \frac{1}{b_2^{p+1}a_0}
\left[a_0\left(b_2b_{6+k}^p-b_{4+k}b_4^p\right)+
b_2^{p}\left(a_2b_{4+k}- a_{4+k}b_2\right)\right].$$

Now, applying  Theorem~\ref{PS}
to the pair
$$(1, \ldots, {\hat{r-3-k}}, \ldots, {\hat{r-1}}, r, r+1)\quad\mbox{and}\quad
(1, \ldots, r, {\hat{r+1}}, r+2),$$
we get 
 $$a_2 b_{4+k}-a_{4+k}b_2 = a_0\cdot \left(\nu_{1, \ldots, 
{\hat {r-3-k}},\ldots, {\hat{r-1}}, r, r+1, r+2}\right).$$

Similarly the pair
 $$(1, \ldots, {\hat{r-3-k}}, \ldots, {\hat{r-1}}, r, {\hat {r+1}}, r+2)
\quad\mbox{and}\quad
(3, \ldots, r+2, {\hat{r+3}}, r+4)$$
gives, 
$$-b_2b_{6+k}^p+b_4^pb_{4+k} -b_2^p\cdot \left(\nu_{1, \ldots, 
{\hat {r-3-k}},\ldots, {\hat{r-1}}, r, r+1, r+2}\right)
+a_0^p\cdot \left(\nu_{1, \ldots, 
{\hat {r-3-k}},\ldots, \{{\hat{r-1}}, \ldots, r+4\}}\right) = 0,$$
where $b_{6+k}^p = 0$ if  $r-5\leq k\leq r-4$.

Therefore
 $$a_0\left[b_2b_{6+k}^p-b_4^pb_{4+k}\right] + 
b_2^p\left[a_2 b_{4+k}-a_{4+k}b_2\right] 
= a_0^{p+1} 
\cdot \left(\nu_{1, \ldots, 
{\hat {r-3-k}},\ldots, \{{\hat{r-1}}, \ldots,  r+4\}}\right).$$
This gives  

$$A^{(1)}_k = \frac{a_0^p}{b_2^{p+1}}B(1,k) = 
\tfrac{B(-1)^p}{B(0)^{p+1}}\cdot B(1,k).$$

\vspace{5pt}

\noindent{\underline{Case}~(2)}.\quad Let $j\geq 1$ and $0\leq k\leq r-2(j+1)-2$. 
Now 
$$B(j-1, 2)B(j, 2+k) = \left(\nu_{1, \ldots, {\hat {r-2j-1}}, \ldots, 
\{{\hat{r-2j+3}}, \ldots, r+2j\}}\right)\cdot 
\left( \nu_{1, \ldots, {\hat {r-2j-3-k}}, \ldots, 
\{{\hat{r-2j+1}}, \ldots, r+2j+2\}}\right)$$
and 
$$B(j-1, 4+k)B(j) = \left(\nu_{1, \ldots, {\hat {r-2j-3-k}}, \ldots, 
\{{\hat{r-2j+3}}, \ldots, r+2j\}}\right)\cdot 
\left( \nu_{1, \ldots, 
\{{\hat{r-2j-1}}, \ldots, r+2j+2\}}\right).$$

Using Theorem~\ref{PS} for the pair 
$$\left(1,  \ldots, {\hat {r-2j-3-k}}, \ldots, {\hat {r-2j-1}}, \ldots, 
\{{\hat{r-2j+3}}, \ldots, r+2j\}\right),$$
$$\left(1, \ldots, \{{\hat{r-2j+1}}, \ldots, r+2j+2\}\right)$$

 we can check
\begin{multline*}B(j-1, 2)B(j, 2+k)
-B(j-1, 4+k)B(j) \\
= \left({\nu}_{1, \ldots, {\hat {r-2j-3-k}}, \ldots, {\hat {r-2j-1}},
 \ldots, \{{\hat {r-2j+3}}, \ldots, r+2j+2\}}\right)\cdot B(j-1).\end{multline*}
On the other hand 
$$B(j)B(j, 4+k)^p = \left( \nu_{1, \ldots, 
\{{\hat{r-2j-1}}, \ldots, r+2j+2\}}\right)\cdot 
\left( \nu_{3, \ldots,{\hat{r-2j-3-k}}, \ldots,
\{{\hat{r-2j+3}}, \ldots, r+2j+4\}}\right)$$
and 
$$B(j, 2+k)B(j,2)^p = \left( \nu_{1, \ldots, {\hat{r-2j-3-k}},\ldots 
\{{\hat{r-2j+1}}, \ldots, r+2j+2\}}\right)\cdot 
\left( \nu_{3, \ldots,{\hat{r-2j-1}}, \ldots,
\{{\hat{r-2j+3}}, \ldots, r+2j+4\}}\right).$$
Now using Theorem~\ref{PS} for the pair
$$ \left(1, \ldots, {\hat{r-2j-3-k}},\ldots 
\{{\hat{r-2j-1}}, \ldots, r+2j+2\}\right),\quad 
\left( 3,  \ldots,
\{{\hat{r-2j+3}}, \ldots, r+2j+4\}\right)$$
we get

\begin{multline*} B(j)B(j, 4+k)^p-
B(j, 2+k)B(j, 2)^p = -B(j)^p\cdot 
\left({\nu}_{1, \ldots, {\hat {r-2j-3-k}}, \ldots, {\hat {r-2j-1}},
 \ldots, \{{\hat {r-2j+3}}, \ldots, r+2j+2\}}\right)\\
+ \left({\nu}_{1, \ldots, {\hat {r-2j-3-k}}, \ldots,
  \{{\hat {r-2j-1}}, \ldots, r+2j+4\}}\right)\cdot B(j-1)^p,
\end{multline*}
where $B(j, 4+k)^p = 0$ if $r-2j-5\leq k \leq r-2j-4$.

Therefore 
\begin{multline*}
B(j)^p\left(B(j-1, 2)B(j, 2+k)
-B(j-1, 4+k)B(j)\right)\\
+B(j-1)\left(B(j)B(j, 4+k)^p-
B(j, 2+k)B(j, 2)^p\right)\\
= B(j-1)^{p+1}\left({\nu}_{1, \ldots, {\hat {r-2j-3-k}}, \ldots,
  \{{\hat {r-2j-1}}, \ldots, r+2j+4\}}\right) 
= B(j-1)^{p+1}B(j+1, k).
\end{multline*}

This implies
$$A^{(j+1)}_k = \left[\tfrac{B(j-1)^{p}}{B(j)^{p+1}}\right]\cdot B(j+1,k).$$
\end{proof}

\section{$r=2s$ is even}

Throughout this section ${\mathcal{K}} = \F_p(x_{ij})$ and   $G= (\Z/p\Z)^r$ acts
on  ${\mathcal{K}}[X, Y, Z]$ as in (\ref{e1}). Here we deal with the case when 
$r= 2s \geq 4$ is an even
integer and $p>3$ is a prime number. The proof for $r=4$ is same as 
in [PS].

In Proposition~\ref{even}, we construct  elements 
$\{f_1, f_1, \ldots, f_{s+2}\}$ in ${\mathcal{K}}[X, Y, Z]^G$ such that
they  have the leading monomials
(in fact the leading terms) as stated in the part~(1) of the conjecture of 
~\cite{CSW}.

Following \cite{CSW} we  define the elements  $g_1$, ${\bar g_2}$ and $g_2$ in ${\mathcal{K}}[X, Y, Z]$.

Let $g_1 := f_{(1, 2, \ldots, r, r+1)}$ and ${\bar g_2}:= f_{(1,2, \ldots, r, r+2)}$.
Then 
$$ g_1 = c_0Y^{p^s}+
c_1\Delta^{p^{s-1}}X^{p^s-2p^{s-1}}
+c_2Y^{p^{s-1}}X^{p^s-p^{s-1}}
+c_3\Delta^{p^{s-2}}X^{p^{s}-2p^{s-2}}+ \ldots + c_rYX^{p^{s}-1}$$

and 
\begin{multline*}{\bar g_2}  = -c_0\Delta^{p^s}+
d_1\Delta^{p^{s-1}}X^{2p^s-2p^{s-1}}
+d_2Y^{p^{s-1}}X^{2p^s-p^{s-1}}
+d_3\Delta^{p^{s-2}}X^{2p^{s}-2p^{s-2}}\\
+ \ldots + d_rYX^{2p^s-1},\end{multline*}
where $c_i, d_i\in {\mathcal{K}}\setminus \{0\}$  are given as in 
Notation~\ref{ne}.
 
Let $f_1 = g_1/c_0$ and let ${\bar f_2} = {\bar g_2}/c_0$. Then 

\begin{multline}\label{f1e}
f_1 = \\Y^{p^s}+
\tfrac{c_1}{c_0}\Delta^{p^{s-1}}X^{p^s-2p^{s-1}}
+\tfrac{c_2}{c_0}Y^{p^{s-1}}X^{p^s-p^{s-1}}
+\tfrac{c_3}{c_0}\Delta^{p^{s-2}}X^{p^{s}-2p^{s-2}}+ \ldots 
+\tfrac{c_r}{c_0}YX^{p^{s}-1}\end{multline}
and 

\begin{multline*}{\bar f_2}  
 = -\Delta^{p^s}+
\tfrac{d_1}{c_0}\Delta^{p^{s-1}}X^{2p^s-2p^{s-1}}
+\tfrac{d_2}{c_0}Y^{p^{s-1}}X^{2p^s-p^{s-1}}
+\tfrac{d_3}{c_0}\Delta^{p^{s-2}}X^{2p^{s}-2p^{s-2}}\\
+ \ldots + \tfrac{d_r}{c_0}YX^{2p^s-1}.\end{multline*}

Let 
\begin{equation}\label{f2e}f_2 = \left(\tfrac{c_0}{2c_1}\right)\cdot 
\frac{{\bar f_2}+f_1^2}{X^{p^s-2p^{s-1}}}.
\end{equation}
Then  
\begin{equation}\label{*1} f_2^p+{\bar f_2}f_1^p = f_2^p+
\left(\tfrac{2c_1}{c_0}X^{p^s-2p^{s-1}}f_2  -f_1^2) \right)f_1^p = 
f_2^p -f_1^pf_1^2 +\tfrac{2c_1}{c_0}X^{p^s-2p^{s-1}}f_2 f_1^p.\end{equation}

\begin{notation}\label{ne}Here $r=2s$. The elements $c_i$ and $d_i$, 
occurring in (\ref{f1e}) and (\ref{f2e}), are  the $r\times r$ minors 
of the matrix $\Gamma$ 
and are given as follows: A subsequence 
$(j_1,  \ldots, j_{r})$ of $(1, 2, r, r+1, r+2)$, 
determines a $r\times r$ minor, $\nu_{j_1,  \ldots, j_{r}}$,
 of the $(2r+2)\times r$ matrix $\Gamma$. For $0\leq m\leq r$
we recall that 

$$c_0=\nu_{1,2,\cdots,r}, \quad 
c_1 = \nu_{1,2,\cdots, {\hat r}, r+1}, \quad \ldots, \quad 
c_m = \nu_{1,2,\cdots, {\hat {r+1-m}},\cdots,  r+1}$$
and $$d_0=\nu_{1,2,\cdots,r}, \quad 
d_1 = \nu_{1,2,\cdots, r-1, r+2}, \quad \ldots, \quad 
d_m = \nu_{1,2,\cdots, {\hat {r+1-m}},\cdots, r, r+2}.$$

For $0\leq m\leq r$ we have 
$c^p_m = \nu_{3,4,\cdots, {\hat {r+3-m}},\cdots, r+2,  r+3}$,
and $c_m = d_m = 0$ if  $m >r$.

Further, let $\{l_1, \ldots, l_{2m+2n+1}\}$ be a set of $2m+2n+1$ 
 integers.
Then we denote
$$\nu_{l_1, l_2, \ldots, 
\{{\hat {l_{2m}}},  \ldots, l_{2m+2n+1}\}} = 
\nu_{l_1, l_2, \ldots, l_{2m-1}, {\hat {l_{2m}}},
l_{2m+1}, {\hat {l_{2m+2}}}, l_{2m+3},{\hat {l_{2m+4}}},
\ldots, l_{2m+2n+1}}.$$

\begin{enumerate}
\item 
For $0\leq j\leq s-2$ and for $0\leq k < r-2j-2$ we denote 
$$C(j, k+1) = \nu_{1,2, \ldots,  
{\hat{r-2j-2-k}}, \ldots,\{{\hat{r-2j}}, \ldots, r+2j+3\}}\neq 0$$
and 
$$C(j) := C(j, 1) = C(j, 0+1)  = \nu_{1,2, \ldots,  
\{{\hat{r-2j-2}}, \ldots, r+2j+3}\}\neq 0$$
and for $j=-1$ we define 
$$C(-1,0) = C(-1, -1+1) = \nu_{1,2,  \ldots, r, {\hat{r+1}}} = c_0$$ and for $k\geq 0$,
$C(-1, k+1) = \nu_{1,2, \ldots,  
{\hat{r-k}}, \ldots, r+1} = c_{k+1}$ and 
$$C(-1) = C(-1, 1) = \nu_{1,2, \ldots, r-1, 
{\hat{r}}, r+1} = c_1.$$

We define $C(j,k+1) = 0$ if $k\geq r-2j-2$.\\

\item For $0\leq j\leq s-2$,  $N_{1j}$, $L_{2j}$ and $N_{2j}$ are the elements of ${\mathcal{K}}$ 
given as
$$N_{1j} =  \tfrac{C(j-1, 2)}{C(j-1)} -\tfrac{C(j,2)^p}{C(j)^p},\quad 
 L_{2j} = \tfrac{C(j-1,3)}{C(j-1)} -\tfrac{C(j,3)^p}{C(j)^p}$$
  
$$\mbox{and}\quad N_{2j} = \tfrac{C(j-1,4)}{C(j-1)}- L_{2j}\cdot \tfrac{C(j,2)}{C(j)} - 
\tfrac{C(j,4)^p}{C(j)^p}.$$

For $j=-1$, let $N_{1,-1} = -{c_2^p}/{c_1^p}$, 
 $$L_{2,-1} = 
  \left(\frac{c_2}{c_1}\right)^p\frac{c_1}{c_0} -
 \left(\frac{c_3}{c_1}\right)^p - \frac{d_1}{c_0}$$
 and 
 $$N_{2,-1} = -L_{2,-1}\cdot \frac{c_{2}}{c_1}  
+\left(\frac{c_2}{c_1}\right)^p\frac{c_2}{c_0} - 
\left(\frac{c_4}{c_1}\right)^p -\frac{d_2}{c_0}.$$

\item We consider the following set of elements in the polynomial ring 
 ${\mathcal{K}}[t_1, \ldots, t_{s+2}]$,
where $E_{1j}({\bf \underline{t}})\in 
{\mathcal{K}}[t_1, \ldots, t_{j+2}]$ and 
$E_{2j}({\bf\underline{ t}})\in 
{\mathcal{K}}[t_1, \ldots, t_{j+3}]$. Here ${\bf \underline{t}} = (t_1, \ldots, t_{s+2})$.

\begin{enumerate}
\item For $0\leq j\leq s-2$, 
$$E_{1j}({\bf\underline{ t}}) = 
t_1^{p^{j+2}-\frac{p}{2}-
\frac{p^{j+1}}{2}-p^j}(t_1t_2\cdots t_{j+2})^{(p-1)/2}t_{j+2}.$$
\item $$E_{2j}({\bf\underline{ t}}) = 
t_1^{p^{j+2}-\frac{p}{2}-\frac{p^{j+2}}{2}-p^{j+1}}(t_1t_2\cdots 
t_{j+3})^{(p-1)/2}t_{j+3}.$$
\end{enumerate}

\item  We consider the following  set of elements $F_{j+3}({\bf\underline{ t}})$ 
in the polynomial ring
${\mathcal{K}}[t_1, t_2, \ldots, t_{s+2}][X]$ where  $F_{j+3}({\bf \underline{t}})\in 
{\mathcal{K}}[t_1, \ldots, t_{j+2}][X]$.

\begin{multline}\label{F3te}
F_3({\bf\underline{ t}}) =  t_2^p-t_1^{p+2}+
\tfrac{2c_1}{c_0}X^{p^s-2p^{s-1}}t_2t_1^p
-\left(\tfrac{c_2}{c_1}\right)^pX^{p^s}t_1^{p+1}\\
 + \left(L_{2,-1}\right)\cdot t_1^{p-1}t_2 X^{2p^s-2p^{s-1}}
+ \left(N_{2,-1}\right)\cdot t_1^{(p-3)/2}t_2^{(p+1)/2}X^{2p^s-p^{s-1}}.\end{multline}

For $0\leq j\leq s-2$, let

\begin{multline}\label{Fj+4e} F_{j+4}({\bf\underline{ t}}) = t_{j+3}^p- 
t_1^{p^j(p^2-1)}t_{j+2}+ \left(N_{1j}\right) \cdot 
E_{1j}({\bf \underline{t}})X^{p^{s-1-j}} \\
+ \left(L_{2j}\right)\cdot
t_1^{p^{j+1}(p-1)}t_{j+3} X^{2p^{s-1-j}-2p^{s-2-j}}
+\left(N_{2j}\right) \cdot 
E_{2j}({\bf \underline{t}}) X^{2p^{s-1-j}-p^{s-1-j}}.\end{multline}

\end{enumerate}

\end{notation}

It is easy to check the following lemma.

\begin{lemma}\label{LMe}Let $r = 2s \geq 4$ be an integer.
Suppose there exists an integer
 $j_1$  and homogeneous polynomials $f_3, \ldots, f_{j_1+3} \in {\mathcal{K}}[X,Y, Z]$ 
such that 
for every $0\leq j\leq j_1\leq s-2$ we have $LM(f_{j+3}) = Y^{p^{s+1+j}+2p^{s-2-j}}$.
Then, for all $0\leq j\leq j_1$, we have the following 
 
\begin{enumerate}
\item $LT(f_1^{p^j(p^2-1)}f_{j+2}) = Y^{p^{s+2+j}+2p^{s-1-j}}$,\\
\item $LT(f_1^{p^{j+1}(p-1)}f_{j+3}) = Y^{p^{s+2+j}+2p^{s-2-j}}$,\\
\item  $LT(E_{1j}({\bf\underline{ f}})) = Y^{p^{s+2+j}+p^{s-1-j}}$,\\
\item $LT(E_{2j}({\bf \underline{f}})) = Y^{p^{s+2+j}+p^{s-2-j}}$.
\end{enumerate}
\end{lemma}

\begin{proposition}\label{even}
Let  $r=2s\geq 4$ and $p>3$ be  a prime number.
Let $f_1, \ldots, f_{s+2}$ be elements in ${\mathcal{K}}[X, Y, Z][X^{-1}]$ which are 
defined as follows:

\begin{align*}
& f_1\text{ and } f_2\text{ are the elements as in (\ref{f1e}) and (\ref{f2e})},\\\\
&  f_3= \left(\tfrac{c_1^{p+1}}{c_0^p\cdot C(0)}\right)
 \frac{F_3({\bf \underline{f}})}{X^{2p^s-2p^{s-2}}},\\\\
& f_{j+4} = \tfrac{C(j)^{p+1}}{(C(j-1)^pC(j+1)}
\frac{F_{j+4}({\bf \underline{f}})}{X^{2p^{s-1-j}-2p^{s-3-j}}},\quad\text{for}\quad 0\leq j
\leq s-3,\\\\
& f_{s+2}= \tfrac{C(s-3)^p}{C(s-2)^p} \cdot
\frac{F_{s+2}({\bf\underline{ f}})}{X^{2p}}.\end{align*}

 Then 
\begin{enumerate}
\item $f_1, \ldots, f_{s+1}, f_{s+2}$ are homogeneous polynomials and belong to 
the ring ${\mathcal{K}}[X, Y, Z]^G$ and
\item $LT(f_1) = Y^{p^s}$, $LT(f_2) = Y^{p^s+2p^{s-1}}$, $LT(f_{s+2}) = Z^{p^r}$
and 
$$LT(f_{j+3}) = Y^{p^{s+1+j}+2p^{s-2-j}},~~\text{ for }~~ 0\leq j\leq s-2.$$
\end{enumerate}
\end{proposition}

\begin{rmk}By Lemma~\ref{ML} it follows that 
$f_3, \ldots, f_{s+2}$ as given in the above proposition are indeed well 
defined elements.
\end{rmk}

\begin{proof}Note that $f_1$, $f_2$ satisfy the conditions of the above proposition.
The proof is similar to the proof  for Proposition~\ref{odd} 
once we make appropriate choices of $f_2$, $f_3$. 

First we write a formula for $f_3$ mod the ideal $I_4 = X^{2p^{s-2}}(X, Y)$ in ${\mathcal{K}}[X,Y,Z]$.
Let $I_2 = X^{2p^{s}}(X, Y)$ and $I_3 = X^{2p^{s-1}}(X, Y)$.
Since $p>3$
by (\ref{f2e})
\begin{multline}\label{f_2*}f_2 \equiv_{I_3}\\ \left(\tfrac{c_0}{2c_1}\right) 
X^{2p^{s-1}}Z^{p^{s}} +
  \Delta^{p^{s-1}}Y^{p^{s}} +
\left(\tfrac{c_2}{c_1}\right)Y^{p^{s}+p^{s-1}}X^{p^{s-1}}+\cdots+
\left(\tfrac{c_r}{c_1}\right)Y^{p^{s}+1}X^{2p^{s-1}-1}\\\\
\equiv_{I_3} \left(\tfrac{c_0}{2c_1}\right) 
X^{2p^{s-1}}Z^{p^{s}} + Y^{p^s}\sum_{l=2}^{s+1}
\left[\tfrac{c_{2l-3}}{c_1}\Delta^{p^{s-l+1}}X^{2p^{s-1}-2p^{s-l+1}}
+\tfrac{c_{2l-2}}{c_1}Y^{p^{s-l+1}}X^{2p^{s-1}-p^{s-l+1}}\right].\end{multline}

Therefore

\begin{multline*}f_2^p \equiv_{I_2} \left(\tfrac{c_0}{2c_1}\right)^p 
X^{2p^{s}}Z^{p^{s+1}}\\
 + Y^{p^{s+1}}\sum_{l=2}^{s+1}
\left[\left(\tfrac{c_{2l-3}}{c_1}\right)^p\Delta^{p^{s-l+2}}X^{2p^{s}-2p^{s-l+2}}
+\left(\tfrac{c_{2l-2}}{c_1}\right)^pY^{p^{s-l+2}}X^{2p^{s}-p^{s-l+2}}\right].\end{multline*}

Now we will use the following  set of equalities 
\begin{enumerate}
\item  $f_1^p\equiv_{I_2} Y^{p^{s+1}}$,\\
\item $f_1^{p+1}X^{p^s} \equiv_{I_2} Y^{p^{s+1}}X^{p^s}\cdot f_1$,\\
\item $f_1^{p-1}f_2 X^{2p^s-2p^{s-1}}\equiv_{I_2} X^{2p^s-2p^{s-1}}Y^{p^{s+1}-p^s}\cdot
f_2$ and \\
\item $f_1^{(p-3)/2}f_2^{(p+1)/2} X^{2p^s-p^{s-1}}\equiv_{I_2} 
X^{2p^s-p^{s-1}}Y^{p^{s+1}+p^{s-1}}$.
\end{enumerate}

We have

\begin{multline*}
{\bar f_2}f_1^p  \equiv_{I_2} 
-\Delta^{p^s}Y^{p^{s+1}}+Y^{p^{s+1}}\sum_{l=3}^{s+2}\left[
\frac{d_{2l-5}}{c_0}\Delta^{p^{s-l+2}}X^{2p^s-2p^{s-l+2}}
 +\frac{d_{2l-4}}{c_0}Y^{p^{s-l+2}}X^{2p^s-p^{s-l+2}}\right].\end{multline*}
 
By (\ref{*1})
$$f_2^p-f_1^{p+2}+ \frac{2c_1}{c_0} f_2 f_1^pX^{p^s-2p^{s-1}} 
= f_2^p+{\bar f_2}f_1^p,$$
therefore

\begin{multline*}
f_2^p-f_1^{p+2}+ \frac{2c_1}{c_0} f_2 f_1^pX^{p^s-2p^{s-1}} 
\equiv_{I_2} 
 \left(\tfrac{c_0}{2c_1}\right)^p X^{2p^s}Z^{p^{s+1}} +
  \left(\tfrac{c_2}{c_1}\right)^p Y^{p^{s+1}+p^s}X^{p^s} \\\\
    + Y^{p^{s+1}}\sum_{l=3}^{s+1}
\left[\left((\tfrac{c_{2l-3}}{c_1})^p+\tfrac{d_{2l-5}}{c_0}\right)\Delta^{p^{s-l+2}}X^{2p^{s}-2p^{s-l+2}}
+\left((\tfrac{c_{2l-2}}{c_1})^p+\tfrac{d_{2l-4}}{c_0}\right)Y^{p^{s-l+2}}X^{2p^{s}-p^{s-l+2}}\right]\\\\
+  Y^{p^{s+1}}\left[\tfrac{d_{r-1}}{c_0}
\Delta X^{2p^s-2}+ \tfrac{d_r}{c_0}
YX^{2p^s-1}\right]. \end{multline*}

Note that 
\begin{multline*}f_1^{p+1}X^{p^s} \equiv_{I_2} 
 Y^{p^{s+1}+p^s}X^{p^s}
 +Y^{p^{s+1}}\sum_{l=3}^{s+2}\left[
\frac{c_{2l-5}}{c_0}\Delta^{p^{s-l+2}}X^{2p^s-2p^{s-l+2}}
 +\frac{c_{2l-4}}{c_0}Y^{p^{s-l+2}}X^{2p^s-p^{s-l+2}}\right].
 \end{multline*}
 
Therefore

\begin{multline*}
f_2^p-f_1^{p+2}+ \frac{2c_1}{c_0} f_2 f_1^pX^{p^s-2p^{s-1}}
-\left(\tfrac{c_2}{c_1}\right)^pX^{p^s}f_1^{p+1} \equiv_{I_2}
\left(\tfrac{c_0}{2c_1}\right)^p X^{2p^s}Z^{p^{s+1}}\\\\
  + Y^{p^{s+1}}\sum_{l=3}^{s+1}
\left[\left(-\left(\tfrac{c_2}{c_1}\right)^p\left(\tfrac{c_{2l-5}}{c_0}\right)+\left(\tfrac{c_{2l-3}}{c_1}\right)^p+\tfrac{d_{2l-5}}{c_0}\right)\Delta^{p^{s-l+2}}X^{2p^{s}-2p^{s-l+2}}\right.\\\\
\left.
+\left(-\left(\tfrac{c_2}{c_1}\right)^p\left(\tfrac{c_{2l-4}}{c_0}\right)+
\left(\tfrac{c_{2l-2}}{c_1}\right)^p+\tfrac{d_{2l-4}}{c_0}\right)Y^{p^{s-l+2}}X^{2p^{s}-p^{s-l+2}}\right]\\\\
+  Y^{p^{s+1}}\left[\left( -\left(\tfrac{c_2}{c_1}\right)^p\left(\tfrac{c_{r-1}}{c_0}\right)+
\tfrac{d_{r-1}}{c_0}\right)
\Delta X^{2p^s-2}+ 
\left(-\left(\tfrac{c_2}{c_1}\right)^p\left(\tfrac{c_r}{c_0}\right)+
\left(\tfrac{d_r}{c_0}\right)\right)
YX^{2p^s-1}\right]. \end{multline*}

By definition $L_{2,{-1}}\in {\mathcal{K}}$ such that 
$$L_{2,{-1}} - \left(\frac{c_2}{c_1}\right)^p\frac{c_1}{c_0} +
 \left(\frac{c_3}{c_1}\right)^p + \frac{d_1}{c_0} = 0$$

and  $N_{2,-1}\in {\mathcal{K}}$ such that
$$N_{2,-1} + L_{2,-1}\cdot \frac{c_{2}}{c_1}  
-\left(\frac{c_2}{c_1}\right)^p\frac{c_2}{c_0}+ \left(\frac{c_4}{c_1}\right)^p+
\frac{d_2}{c_0} = 0.$$

Therefore  for 
\begin{multline*}
F_3({\bf\underline{f}}) = 
f_2^p-f_1^{p+2}+
\frac{2c_1}{c_0}X^{p^s-2p^{s-1}}f_2f_1^p
-\left(\frac{c_2}{c_1}\right)^pX^{p^s}f_1^{p+1}\\\\
 + \left(L_{2,-1}\right)\cdot f_1^{p-1}f_2 X^{2p^s-2p^{s-1}}
+ \left(N_{2,-1}\right)\cdot f_1^{(p-3)/2}f_2^{(p+1)/2}X^{2p^s-p^{s-1}}\end{multline*}
we have 

\begin{multline*}
F_3({\bf\underline{f}})   \equiv_{I_2} 
\left(\frac{c_0}{2c_1}\right)^p X^{2p^s}Z^{p^{s+1}}\\\\
  + Y^{p^{s+1}}\sum_{l=4}^{s+1}
\left[\left(L_{2,-1}\cdot \tfrac{c_{2l-5}}{c_1}-(\tfrac{c_2}{c_1})^p
\tfrac{c_{2l-5}}{c_0}+(\tfrac{c_{2l-3}}{c_1})^p+
\tfrac{d_{2l-5}}{c_0}\right)\Delta^{p^{s-l+2}}X^{2p^{s}-2p^{s-l+2}}\right.\\\\
\left.
+\left(L_{2,-1}\cdot \tfrac{c_{2l-4}}{c_1}
-(\tfrac{c_2}{c_1})^p\tfrac{c_{2l-4}}{c_0}+
(\tfrac{c_{2l-2}}{c_1})^p+\tfrac{d_{2l-4}}{c_0}\right)
Y^{p^{s-l+2}}X^{2p^{s}-p^{s-l+2}}\right]\\\\
+  Y^{p^{s+1}}\left[
\left(L_{2,-1}\cdot \tfrac{c_{r-1}}{c_1}
-(\tfrac{c_2}{c_1})^p\tfrac{c_{r-1}}{c_0}+
\tfrac{d_{r-1}}{c_0}\right)
\Delta X^{2p^s-2}+ 
\left(L_{2,-1}\cdot \tfrac{c_r}{c_1}
-(\tfrac{c_2}{c_1})^p\tfrac{c_r}{c_0}+
\tfrac{d_r}{c_0}\right)
YX^{2p^s-1}\right]. \end{multline*}

Therefore  (recall $C(0) = C(0,1)$ and 
$c_1 = C(-1)$) by Lemma~\ref{ML}

\begin{multline*} 
F_3({\bf\underline{f}})   \equiv_{I_2} \left(\frac{c_0}{2c_1}\right)^p X^{2p^s}Z^{p^{s+1}}\\\\
  + Y^{p^{s+1}}\left(\tfrac{c_0^p}{c_1^{p+1}}\right)
  \sum_{l=4}^{s+1}
\left[C(0,2l-7)\Delta^{p^{s-l+2}}X^{2p^{s}-2p^{s-l+2}}
+ C(0, 2l-6)Y^{p^{s-l+2}}X^{2p^{s}-p^{s-l+2}}\right]\\\\
+  Y^{p^{s+1}}
\left(\tfrac{c_0^p}{c_1^{p+1}}\right)
\left[C(0, r-3)
\Delta X^{2p^s-2}+ C(0,r-2)
YX^{2p^s-1}\right]. \end{multline*}

Since

$$f_3 = \left(\tfrac{c_1^{p+1}}{c_0^p\cdot C(0)}\right)
 \frac{1}{X^{2p^s-2p^{s-2}}}F_3({\bf\underline{f}}),$$
by Lemma~\ref{ML}  for   $I_4 = X^{2p^{s-2}}(X, Y)$, we have

\begin{multline}\label{*f3}f_3\equiv_{I_4}\\\
 \tfrac{C(-1)}{2^pC(0)} X^{2p^{s-2}}Z^{p^{s+1}}
+ Y^{p^{s+1}}\sum_{l=2}^{s}\left[\tfrac{C(0,2l-3)}{C(0)}
\Delta^{p^{s-l}}X^{2p^{s-2}-2p^{s-l}}
+\tfrac{C(0,2l-2)}{C(0)}Y^{p^{s-l}}X^{2p^{s-2}-p^{s-l}}\right].
\end{multline}

\vspace{5pt}

\noindent{\bf Claim}.\quad
Let $I_{j+4} = X^{2p^{s-2-j}}(X, Y)$ and 
$I_{j+3} = X^{2p^{s-1-j}}(X, Y)$ denote the ideals in  ${\mathcal{K}}[X, Y, Z]$. Then 
for $0\leq j\leq s-2$ (here $C(j) = C(j ,1)$)

\begin{multline}\label{f_{j+3}e} f_{j+3}   \equiv_{I_{j+4}}  
\left(\tfrac{C(j-1)}{2^{p^{j+1}}C(j)}\right)X^{2p^{s-2-j}}Z^{p^{s+1+j}}\\\
+ Y^{p^{s+1+j}}\sum_{l=2}^{s-j}\left[\tfrac{C(j,2l-3)}{C(j)}
\Delta^{p^{s-l-j}}X^{2p^{s-2-j}-2p^{s-l-j}}
+\tfrac{C(j,2l-2)}{C(j)}Y^{p^{s-l-j}}X^{2p^{s-2-j}-p^{s-l-j}}\right].
\end{multline}

In particular $LT(f_{j+3}) = Y^{p^{s+1+j}+2p^{s-2-j}}$ and $f_{j+3}$ is a 
homogeneous polynomial in ${\mathcal{K}}[X, Y, Z]^G$.
\vspace{5pt}

\noindent{\underline{Proof of the claim}}:\quad 
We have already proved the claim for  $j=0$.

\vspace{5pt}

Now for the rest of the proof we will use the following set 
of equalities which can be checked easily by inducting on $j$ and 
(\ref{f_{j+3}e}).

For $0\leq j \leq s-2$, 
\begin{enumerate}
\item[Eq.(1)] $f_1^{p^j(p^2-1)}f_{j+2} \equiv_{I_{j+3}} (Y^{p^s})^{p^j(p^2-1)} f_{j+2}$.\\
\item[Eq.(2)] $E_{1j}({\bf\underline{ f}})X^{p^{s-1-j}} \equiv_{I_{j+3}}X^{p^{s-1-j}}
Y^{p^{s+2+j}+p^{s-1-j}}$.\\
\item[Eq.(3)] $f_1^{p^{j+1}(p-1)}f_{j+3} X^{2p^{s-1-j}-2p^{s-2-j}}\equiv_{I_{j+3}}
Y^{p^{s+2+j}-p^{s+1+j}}X^{2p^{s-1-j}-2p^{s-2-j}}f_{j+3}$.\\
\item[Eq.(4)] $E_{2j}({\bf\underline{ f}}) X^{2p^{s-1-j}-p^{s-2-j}} \equiv_{I_{j+3}}
Y^{p^{s+2+j}+p^{s-2-j}}X^{2p^{s-1-j}-p^{s-2-j}}$.
\end{enumerate}

\vspace{5pt}

Let $0\leq j\leq s-3$.

We assume 
the expression~(\ref{f_{j+3}e}) holds for 
 $f_3, \ldots, f_{j+3}$.  In particular 
$LT(f_{j_1+3}) = Y^{p^{s+j_1+1}+2p^{s-2-j_1}}$, for $0\leq j_1\leq j$.
Here we prove that the same holds  for  
$f_{j+4}$.
Since it holds for $j = 0$, the claim will follow by induction on $j$.

Evaluating $F_{j+4}({\bf\underline{ t}})$ at ${\bf\underline{t}} = 
{\bf\underline{ f}}$, where ${\bf\underline{f}} = (f_1, f_2, \ldots, f_{j+3})$ we get 

\begin{multline}\label{Fj+4} F_{j+4}({\bf \underline{f}}) := f_{j+3}^p- 
f_1^{p^j(p^2-1)}f_{j+2}+ \left(N_{1j}\right) \cdot 
E_{1j}({\bf\underline{ f}})X^{p^{s-1-j}} \\
+ \left(L_{2j}\right)\cdot
f_1^{p^{j+1}(p-1)}f_{j+3} X^{2p^{s-1-j}-2p^{s-2-j}}
+\left(N_{2j}\right) \cdot 
E_{2j}({\bf\underline{ f}}) X^{2p^{s-1-j}-p^{s-2-j}}.\end{multline}

Since $I_{j+3} = X^{2p^{s-1-j}-2p^{s-3-j}}I_{j+5}$, it
 is enough to prove the following equality

\begin{multline}\label{*11}
F_{j+4}({\bf \underline{f}})   \equiv_{I_{j+3}}  
\left(\tfrac{C(j-1)}{2^{p^{j+1}}C(j)}\right)^p X^{2p^{s-1-j}}Z^{p^{s+2+j}}\\\\
+\tfrac{C(j-1)^p}{C(j)^{p+1}} Y^{p^{s+2+j}}
\sum_{l=2}^{s-1-j}
\left[C(j+1, 2l-3)
\Delta^{p^{s-l-1-j}}X^{2p^{s-1-j}-2p^{s-l-1-j}}\right.\\\\
\left.+
C(j+1,2l-2)
Y^{p^{s-l-1-j}}X^{2p^{s-1-j}-p^{s-l-1-j}}\right].
\end{multline}

Using the induction hypothesis on (\ref{f_{j+3}e}) we get 
\begin{multline*} 
f^p_{j+3} \equiv_{I_{j+3}}  \left(\tfrac{C(j-1)}{2^{p^{j+1}}C(j)}\right)^p
X^{2p^{s-1-j}}Z^{p^{s+2+j}}\\\\
+ Y^{p^{s+2+j}}\sum_{l=2}^{s-j}\left[\tfrac{C(j,2l-3)^p}{C(j)^p}
\Delta^{p^{s-l-j+1}}X^{2p^{s-1-j}-2p^{s-l-j+1}}
+\tfrac{C(j,2l-2)^p}{C(j)^p}Y^{p^{s-l-j+1}}X^{2p^{s-1-j}-p^{s-l-j+1}}\right].
\end{multline*}

By Eq.(1)  and (\ref{f_2*})
\begin{multline*}f_1^{p^2-1}f_2 \equiv_{I_3}
  \Delta^{p^{s-1}}Y^{p^{s+2}}\\\ +
\left(\tfrac{C(-1,2)}{C(-1)}\right)Y^{p^{s+2}+p^{s-1}}X^{p^{s-1}}+\cdots+
\left(\tfrac{C(-1,r)}{C(-1)}\right)Y^{p^{s+2}+1}X^{2p^{s-1}-1}.\end{multline*}

If $j\geq 1$ then 
by Eq.(1) and the induction hypothesis on (\ref{f_{j+3}e})  we get

\begin{multline*} 
f_1^{p^j(p^2-1)}f_{j+2}  \equiv_{I_{j+3}} Y^{p^{s+j+2}-p^{s+j}}f_{j+2} 
\equiv_{I_{j+3}}\\\\
Y^{p^{s+2+j}}\sum_{l=2}^{s-j}\left[\tfrac{C(j-1,2l-3)}{C(j-1)}
\Delta^{p^{s-l-j+1}}X^{2p^{s-1-j}-2p^{s-l-j+1}}
+\tfrac{C(j-1,2l-2)}{C(j-1)}Y^{p^{s-l-j+1}}X^{2p^{s-1-j}-p^{s-l-j+1}}\right]\\\\
+ Y^{p^{s+2+j}} \left[\tfrac{C(j-1,r-2j-1)}{C(j-1)}\Delta X^{2p^{s-1-j}-2}+
\tfrac{C(j-1,r-2j)}{C(j-1)}YX^{2p^{s-1-j}-1}\right].\end{multline*}

 By definition, for $j\geq 0$,  $N_{1j}\in {\mathcal{K}}$ such that
\begin{equation}\label{n1j} 
N_{1j}  + \tfrac{C(j,2)^p}{C(j)^p} - \tfrac{C(j-1,2)}{C(j-1)} = 0.\end{equation}

Therefore, for $j\geq 0$
\begin{multline}\label{uej+4}f_{j+3}^p - f_1^{p^j(p^2-1)}f_{j+2}+ 
N_{1j}\cdot 
E_{1j}({\bf \underline{f}})X^{p^{s-1-j}}
\equiv_{I_{j+3}}
\left(\tfrac{C(j-1)}{2^{p^{j+1}}C(j)}\right)^p   X^{2p^{s-1-j}}Z^{p^{s+2+j}}\\\\
+Y^{p^{s+2+j}}\sum_{l=3}^{s-j}\left[\left(\tfrac{C(j,2l-3)^p}{C(j)^p}-
\tfrac{C(j-1,2l-3)}{C(j-1)}\right)
\Delta^{p^{s-l-j+1}}X^{2p^{s-1-j}-2p^{s-l-j+1}}\right.\\\\
\left.+\left(\tfrac{(C(j,2l-2)^p}{C(j)^p}-\tfrac{C(j-1,2l-2)}{C(j-1)}
\right)Y^{p^{s-l-j+1}}X^{2p^{s-1-j}-p^{s-l-j+1}}\right]\\\\
+ Y^{p^{s+2+j}}\left[-\tfrac{C(j-1,r-2j-1)}{C(j-1)} 
\Delta X^{2p^{s-1-j}-2}-\tfrac{C(j-1,r-2j)}{C(j-1)} 
Y X^{2p^{s-1-j}-1}\right].\end{multline}

By definition  $L_{2j}$ and $N_{2j}$ are in the field 
${\mathcal{K}}$ such that 

\begin{equation}\label{l2j} 
L_{2j}  + \tfrac{C(j,3)^p}{C(j)^p} - \tfrac{C(j-1,3)}{C(j-1)} = 0.\end{equation}

\begin{equation}\label{n2j} 
N_{2j}  +  L_{2j}\cdot \tfrac{C(j,2)}{C(j)}+\tfrac{C(j,4)^p}{C(j)^p}
 - \tfrac{C(j-1,4)}{C(j-1)} = 0.\end{equation}

This gives 
\begin{multline*}
 F_{j+4}({\bf \underline{f}})\equiv_{I_{j+3}}
\left(\tfrac{C(j-1)}{2^{p^{j+1}}C(j)}\right)^p   X^{2p^{s-1-j}}Z^{p^{s+2+j}}\\\\
+Y^{p^{s+2+j}}\sum_{l=4}^{s-j}\left[\left(L_{2j}\cdot \tfrac{C(j, 2l-5)}{C(j)}+
 \tfrac{C(j,2l-3)^p}{C(j)^p}-
\tfrac{C(j-1,2l-3)}{C(j-1)}\right)
\Delta^{p^{s-l-j+1}}X^{2p^{s-1-j}-2p^{s-l-j+1}}\right.\\\\
\left.+\left(L_{2j}\cdot \tfrac{C(j, 2l-4)}{C(j)}+
\tfrac{(C(j,2l-2)^p}{C(j)^p}-\tfrac{C(j-1,2l-2)}{C(j)}\right)
Y^{p^{s-l-j+1}}X^{2p^{s-1-j}-p^{s-l-j+1}}\right]\\\\
+ Y^{p^{s+2+j}}\left[
\left(L_{2j}\cdot \tfrac{C(j, r-2j-3)}{C(j)}-\tfrac{C(j-1,r-2j-1)}{C(j-1)} 
\right)\Delta X^{2p^{s-1-j}-2}\right.\\\\
\left.+\left(L_{2j}\cdot \tfrac{C(j, r-2j-2)}{C(j)}
-\tfrac{C(j-1,r-2j)}{C(j-1)}\right)Y X^{2p^{s-1-j}-1}\right].\end{multline*}

Now the formula (\ref{*11}) for $F_{j+4}({\bf \underline{f}})$ 
follows from Lemma~\ref{ML}. This proves the claim.

\vspace{5pt}

Now we have homogeneous polynomials  $f_1, \ldots , f_{s+1}$
in ${\mathcal{K}}[X, Y, Z]^G$ with the leading terms as in assertion~(2) of the
proposition. It only remains to  prove the assertion~(1) and (2)  
for $f_{s+2}$.

\vspace{5pt}

We have 
$$f_{s+1}\equiv_{I_{s+2}} \tfrac{C(s-3)}{2^{p^{s-1}}C(s-2)} X^2Z^{p^{2s-1}}+
\Delta Y^{p^{2s-1}}+ \tfrac{C(s-2,2)}{C(s-2)} Y^{p^{2s-1}+1}X.$$ 

Therefore for $I_{s+1} = X^{2p}(X, Y)$
$$f^p_{s+1}\equiv_{I_{s+1}}\left(\tfrac{C(s-3)}{2^{p^{s-1}}C(s-2)}\right)^p X^{2p}Z^{p^{2s}}+
\Delta^p Y^{p^{2s}}+ \tfrac{C(s-2,2)^p}{C(s-2)^p}Y^{p^{2s}+p}X^p.$$ 

Now we follow the above  procedure for $j=s-2$ and consider 
 
$$\begin{array}{lcl}
F_{s+2}({\bf \underline{f}}) & = & f_{s+1}^p -
 f_1^{p^{s-2}(p^2-p)}f_s  + 
\left(N_{1,{s-2}}\right)\cdot E_{1,{s-2}}({\bf\underline{ f}})X^p \\\\
& & + 
\left(L_{2,{s-2}}\right)\cdot f_1^{p^{s-1}(p-1)}f_{s+1}X^{2p-2}+ 
\left(N_{2,{s-2}}\right)\cdot E_{2,{s-2}}({\bf \underline{f}})X^{2p-1}.\end{array}$$

Note, by Eq.(1)
$$\begin{array}{lcl}f_1^{p^{s-2}(p^2-1)}f_s & \equiv_{I_{s+1}} & 
Y^{p^{2s}-p^2}f_s\equiv_{I_{s+}}
\Delta^{p}Y^{p^{2s}}+
\tfrac{C(s-3,2)}{C(s-3)}Y^{p^{2s}+p}X^p \\\\
& & + \tfrac{C(s-3,3)}{C(s-3)}\Delta Y^{p^{2s}}X^{2p-2} + 
\tfrac{C(s-3,4)}{C(s-3)}
Y^{p^{2s}+1}X^{2p-1}\end{array}$$
and 
$E_{1,{s-2}}({\bf\underline{ f}})X^p \equiv_{I_{s+1}} Y^{p^{2s}+p}X^p$. 

By definition,
we have $N_{1,{s-2}}\in {\mathcal{K}}$
such that
\begin{equation*}
N_{1,{s-2}} +\tfrac{C(s-2,2)^p}{C(s-2)^p}-\tfrac{C(s-3,2)}{C(s-3)}
= 0.\end{equation*}
Therefore 

\begin{multline*}
 f_{s+1}^p - f_1^{p^{s-2}(p^2-p)}f_s  + 
\left(N_{1,{s-2}}\right)\cdot E_{1,{s-2}}({\bf\underline{ f}})X^p
\equiv_{I_{s+1}}
\left(\tfrac{C(s-3)}{2^{p^{s-1}}C(s-2)}\right)^p   X^{2p}Z^{p^{2s}}\\\\
+Y^{p^{2s}}
\left[-\tfrac{C(s-3,3)}{C(s-3)} 
\Delta X^{2p-2}-\tfrac{C(s-3,4)}{C(s-3)} 
Y X^{2p-1}\right].\end{multline*}

Now 
$$f_1^{p^{s-1}(p-1)}f_{s+1}X^{2p-2}\equiv_{I_{s+1}}
\Delta X^{2p-2}Y^{p^{2s}} + \tfrac{C(s-2,2)}{C(s-2)} 
Y^{p^{2s}+1} X^{2p-1}$$
and 
$$E_{2,{s-2}}({\bf \underline{f}})X^{2p-1}
\equiv Y^{p^{2s}+1}X^{2p-1}.$$

By definition
  $L_{2,{s-2}}\in {\mathcal{K}}$ such that 
\begin{equation*}L_{2,{s-2}} + \tfrac{C(s-2,3)^p}{C(s-2)^p}-
\tfrac{C(s-3,3)}{C(s-3)} = 
L_{2,{s-2}} -
\tfrac{C(s-3,3)}{C(s-3)} = 0\end{equation*}

and  $N_{2,{s-2}}\in {\mathcal{K}}$ such that 
\begin{equation*}
N_{2,{s-2}}+L_{2,{s-2}}\cdot \tfrac{C(s-2,2)}{C(s-2)} -\tfrac{C(s-3,4)}{C(s-3)}= 
0.\end{equation*}

This gives

$$F_{s+2}({\bf \underline{f}})\equiv_{I_{s+1}}
\left(\tfrac{C(s-3)}{2^{p^{s-1}}C(s-2)}\right)^p X^{2p}Z^{p^{2s}}.$$
By definition 
$$f_{s+2} = \tfrac{C(s-3)^p}{C(s-2)^p} \cdot
\frac{F_{s+2}({\bf\underline{ f}})}{2^{p^s}X^{2p}} \quad\mbox{which implies}\quad f_{s+2} \equiv_{I_{s+3}} Z^{p^{2s}}.$$
This proves the proposition.
\end{proof}

Now we prove the lemma which played a crucial role for the  induction process
 in Proposition~\ref{even}, that is, to construct  $f_{j+4}$ from 
$f_1, \ldots, f_{j+3}$.

Moreover it gives an explicit 
 formula  of $f_{j+4}$, mod the ideal  $X^{2p^{s-3-j}}(X, Y)$,
in terms of the minors $B(j+1,k)$.

\begin{lemma}\label{ML} For every $0\leq j\leq s-2$ and every integer 
$0 \leq k\leq r-2(j+1)-3$, 
$$ B_k^{(j+1)}:= L_{2j}\cdot \tfrac{C(j,3+k)}{C(j)} + 
\tfrac{C(j,5+k)^p}{C(j)^p}-\tfrac{C(j-1,5+k)}{C(j-1)} 
= \tfrac{C(j-1)^p C(j+1,k+1)}{C(j)^{p+1}},$$
where 
$L_{2j} = \tfrac{C(j-1,3)}{C(j-1)} - \tfrac{C(j,3)^p}{C(j)^p}$. 
Moreover 
$$B_k^{(0)}:= L_{2,-1}\cdot \tfrac{c_{3+k}}{c_1}  
-\left(\tfrac{c_2}{c_1}\right)^p\tfrac{c_{3+k}}{c_0}+ 
\left(\tfrac{c_{k+5}}{c_1}\right)^p+
\tfrac{d_{3+k}}{c_0} 
 =  \tfrac{c_0^p}{c_1^{p+1}}\cdot C(0, k+1),$$

where $$L_{2,-1} = \left(\tfrac{c_2}{c_1}\right)^p\tfrac{c_1}{c_0} -
 \left(\tfrac{c_3}{c_1}\right)^p - \tfrac{d_1}{c_0}.$$

In particular $B_k^{(j+1)} \neq 0$ if
$0\leq k\leq r-2(j+1)-3$.
\end{lemma} 
\begin{proof}First we prove the assertion for $B_k^{(0)}$.

Applying Theorem~\ref{PS}, for the pair 
$(1, \ldots, {\hat {r+1-k}}, \ldots, {\hat{r+1}}), \quad (3, 4, \ldots, r+3)$
one gets, for all $1\leq k \leq r$

$$\begin{array}{lcl}
0 & =  & (-1)^{r-k+1}\left(\nu_{1,2, \ldots, {\hat{r+1-k}}, \ldots,  
r, r+1-k}\right)\cdot\left( 
\nu_{3,4, \ldots, {\hat{r+1-k}}, \ldots, r+3}\right)\\\\
& & + (-1)^{r+1}\left(\nu_{1,2, \ldots, {\hat{r+1-k}}, \ldots,  
r, r+1}\right)\cdot\left( \nu_{3,4, \ldots, r, r+2, r+3}\right)\\\\
& & + (-1)^{r+2}\left(\nu_{1,2, \ldots, {\hat{r+1-k}}, \ldots,  
r, r+2}\right)\cdot\left( \nu_{3,4, \ldots, r, r+1, r+3}\right)\\\\
 & & +(-1)^{r+3}\left(\nu_{1,2, \ldots, {\hat{r+1-k}}, \ldots,  
r, r+3}\right)\cdot\left(\nu_{3,4, \ldots,  r, r+1, r+2}\right),
\end{array}$$

which is
\begin{equation}\label{B_k0}
 0 = c_0\cdot c_{k+2}^p-c_k\cdot c_{2}^p
+d_k\cdot c_{1}^p-\left(\nu_{1,2, \ldots, {\hat{r+1-k}}, \ldots,  
r, r+3}\right)\cdot c_0^p,\end{equation}

Applying for $k=1$
we get
$$L_{2,-1} = -\frac{1}{c_1^pc_0}\left[c_2^pc_1-c_3^pc_0-d_1c_1^p\right]=-
\frac{1}{c_1^pc_0}\left[c_0^p\cdot \nu_{1,2, \ldots, {\hat{r+1-k}}, \ldots,  r, r+3}\right].$$

Now for $0\leq k <r-2$

$$\begin{array}{lcl}\frac{c_1^{p+1}}{c_0^p)}B_k^{(0)} & = & 
\left(\frac{c_1}{c_0}\right)^pL_{2,-1}\cdot c_{k+3}
+\frac{c_1}{c_0^{p+1}}\left(c_0c_{k+5}^p-c_2^pc_{k+3}+d_{k+3}c_1\right)\\
& = & \left(\frac{c_1}{c_0}\right)^pL_{2,-1}\cdot c_{k+3}
- \frac{c_1}{c_0}\left(\nu_{1,2,\ldots, {\hat{r-k-2}},\ldots,
r,r+3}\right)\\
& = &  -\frac{c_{k+3}}{c_0}\nu_{1,2,\ldots, r-1,r+3}
-\frac{c_1}{c_0}\left(\nu_{1,2,\ldots, {\hat{r-k-2}},\ldots,
r,r+3}\right).\end{array}$$

Applying Theorem~\ref{PS}, for the pair 
$$(1, \ldots, {\hat {r-2-k}}, \ldots, {\hat{r}}, r+1),~(1, 2, \ldots, r, 
{\hat{r+1}}, {\hat{r+2}}, r+3),$$
we get 
$$c_0\cdot \nu_{1, 2, \ldots, {\hat{r-2-k}}, \ldots, \{r, \ldots, r+3\}}
 + c_{k+3}\cdot \nu_{1,2,\ldots, r-1,r+3}
+ c_1\cdot \left(\nu_{1,2,\ldots, {\hat{r-k-2}},\ldots,
r,r+3}\right) = 0,$$
since  $C(0, k+1) = c_0\cdot \nu_{1, 2, 
\ldots, {\hat{r-2-k}}, \ldots, \{\hat{r}, \ldots, r+3\}}$ this proves the identity
for $B_k^{(0)}$.

For $0\leq j\leq s-2$, the term $B_k^{(j+1)}$ can be rewritten as 

\begin{multline}\label{BKj}
B_k^{(j+1)} = \tfrac{1}{C(j)^{p+1}}\Bigl(C(j)C(j, 5+k)^p-
C(j,3)^pC(j, 3+k)\Bigr)\\
 + \tfrac{1}{C(j)C(j-1)}\Bigl(C(j-1,3)C(j, 3+k)-
C(j)C(j-1, 5+k)\Bigr).\end{multline}

\vspace{5pt}

\noindent{\underline{Case}}~(1).\quad  Let $j=0$ and $0\leq k\leq r-5$, then 
 
\begin{multline*}B_k^{(1)} = 
\tfrac{1}{C(0)^{p+1}}\Bigl(C(0)C(0, 5+k)^p-
C(0,3)^pC(0, 3+k)\Bigr)\\
 + \tfrac{1}{C(0)c_1}\Bigl(c_3C(j, 3+k)-
C(0)c_{5+k}\Bigr).\end{multline*}

Now 
$$C(0, 3+k)c_3 = \left(\nu_{1, 2, \ldots, {\hat{r-4-k}}, \ldots, 
\{{\hat{r}}, \ldots, r+3\}}\right) \left(\nu_{1, 2, \ldots, {\hat{r-2}}, 
r-1, r, r+1}\right)$$
and $$C(0)c_{5+k} = \left(\nu_{1, 2, \ldots, {\hat{r-2}}, \ldots, 
\{{\hat{r}}, \ldots, r+3\}}\right) \left(\nu_{1, 2, \ldots, {\hat{r-4-k}}, \ldots, 
r, r+1}\right).$$
Consider the following  two tuples  
$$(1, 2, \ldots, {\hat{r-4-k}}, \ldots, 
{\hat{r-2}}, r-1, r, r+1) \quad{and}\quad  (1, 2, \ldots, r-1, {\hat{r}}, 
 r+1, r+3)$$
of lengths  $r-1$ and $r+1$ respectively.
Using Theorem~\ref{PS}, we get 
$$C(0,3+k)c_3-
C(0)c_{5+k} = c_1\cdot 
\nu_{1, 2, \ldots, {\hat{r-4-k}}, \ldots, {\hat{r-2}}, 
r-1, r, r+1, {\hat{r+2}}, r+3}.$$
On the other hand 
$$C(0)C(0,5+k)^{p}  =   \left(\nu_{1, \ldots, {\hat {r-2}}, r-1, \{{\hat {r}},
\ldots, r+3\}}\right)\left(\nu_{3, \ldots, {\hat{r-4-k}}, \ldots, r, r+1, 
\{{\hat{r+2}}, \ldots, r+5\}}\right),$$
where  $C(0)C(0,5+k)^{p}= 0$, if $k\in \{r-6, r-5\}$.
Note 
$$C(0, 3+k)C(0,3)^p = 
\left(\nu_{1, \ldots, {\hat {r-4-k}}, \ldots, \{{\hat {r}},
\ldots, r+3\}}\right)\left(\nu_{3, \ldots, {\hat{r-2}}, r-1 r, r+1, 
\{{\hat{r+2}}, \ldots, r+5\}}\right).$$

Now applying Theorem~\ref{PS} to the pair of $r-1$-tuple and $r+1$-tuple

$$(1,  \ldots, {\hat{r-4-k}}, \ldots, r-3,  
\{{\hat{r-2}}, \ldots,  r+3\})~~~~\mbox{and}~~~ (3, \ldots, r+1, 
\{{\hat{r+2}},\ldots,   r+5\})$$
we get
\begin{multline*}
C(0)C(0,5+k)^p -C(0,3+k)C(0,3)^p  =   
 c_1^p\cdot \nu_{1, \ldots, {\hat{r-4-k}}, \ldots, 
\{{\hat{r-2}}, \ldots,   r+5\}}\\
-C(0)^p\cdot \nu_{1, \ldots, {\hat{r-4-k}}, \ldots, 
{\hat{r-2}},r-1, r, r+1, {\hat{r+2}},   r+3}.\end{multline*}

Therefore 
$$ B_k^{(1)} = \tfrac{c_1^p}{C(0)^{p+1}}\cdot \nu_{1, \ldots, {\hat{r-4-k}}, 
\ldots, \{{\hat{r-2}}, \ldots,   r+5\}} = 
\tfrac{c_1^p}{C(0)^{p+1}}\cdot C(1, k+1).$$

\vspace{5pt}

\noindent{\underline{Case}}~(2).\quad Let $1\leq j\leq s-2$ and $0\leq k\leq r-2(j+1)-3$.  Then

$$C(j,3+k) C(j-1, 3) = \left(\nu_{1,2 \ldots, {\hat{r-2j-4-k}}, 
\ldots,\{{\hat{r-2j}}, \ldots, r+2j+3\}}\right) \left( \nu_{1,2 \ldots, 
{\hat{r-2j-2}}, \ldots,\{{\hat{r-2j+2}}, \ldots, r+2j+1\}}\right) $$ and 
$$C(j)C(j-1, 5+k) = \left( \nu_{1,2 \ldots, {\hat{r-2j-2}}, r-2j-1, 
\{{\hat{r-2j}}, \ldots, r+2j+3\}}\right)\left(\nu_{1,2 \ldots, {\hat{r-2j-k-4}}, 
\ldots,\{{\hat{r-2j+2}}, \ldots, r+2j+1\}}\right). $$

On the other hand applying Theorem~\ref{PS} for the 
$r-1$-tuple and $r+1$ tuple  
$$(1, 2, \ldots, {\hat{r-2j-k-4}}, \ldots, {\widehat{r-2j-2}}, \ldots, 
 \{{\hat{r-2j+2}}, \ldots, r+2j+1\})$$ and 
$(1,2 \ldots, r-2j-1,  \{{\hat{r-2j}}, \ldots, r+2j+3\})$
gives 
$$\begin{array}{lcl}
0 & = & 
\left(\nu_{1,2, \ldots,{\hat{r-2j-4-k}},\ldots,  
 \{{\hat{r-2j}}, \ldots, r+2j+3\}}\right)\cdot\left( 
\nu_{1,2 \ldots, {\hat{r-2j-2}}, \ldots,\{{\hat{r-2j+2}}, \ldots, r+2j+1\}}\right)\\\\
& & -\left(\nu_{1,2 \ldots, {\hat{r-2j-2}}, r-2j-1, \{{\hat{r-2j}}, \ldots, r+2j+3\}}\right)
\cdot 
\left(\nu_{1,2 \ldots, {\hat{r-2j-k-4}}, \ldots,\{{\hat{r-2j+2}}, \ldots, r+2j+1\}}\right)\\\\
 & & 
 -\left(\nu_{1,2 \ldots, {\hat{r-2j-4-k}}, \ldots, {\hat{r-2j-2}}, \ldots, 
r-2j+1, \{{\hat{r-2j+2}}, \ldots, r+2j+3\}}\right)
\cdot 
\left(\nu_{1,2 \ldots,  r-2j-1, \{{\hat{r-2j}}, \ldots, r+2j+1\}}\right),\end{array}$$
where, by definition
$$C(j-1) = \left(\nu_{1,2 \ldots,  r-2j-1, 
\{{\hat{r-2j}}, \ldots, r+2j+1\}}\right).$$
Hence 
\begin{multline}\label{BKj1} C(j, 3+k)C(j-1,3)- C(j)C(j-1, 5+k)\\
 = \left( 
\nu_{1,2 \ldots, {\hat{r-2j-4-k}}, \ldots, 
{\hat {r-2j-2}}, \ldots, \{{\hat{r-2j+2}}, \ldots, r+2j+3\}}\right)
\cdot C(j-1).\end{multline}

On the other hand 
$$\begin{array}{lcl}
C(j)C(j,5+k)^p & = & 
\left(\nu_{1,2 \ldots, {\hat{r-2j-2}}, r-2j-1, \{{\hat{r-2j}}, \ldots, r+2j+3\}}\right)
\cdot\left(\nu_{3,4 \ldots, {\hat{r-2j-4-k}}, \ldots , \{{\hat{r-2j+2}}, \ldots, 
r+2j+5\}}\right),\\\
& = & 0\quad\mbox{if}\quad k\in \{r-2j-6, r-2j-5\}\end{array}$$
and
$$C(j,3+k)C(j,3)^p = \left(\nu_{1,2 \ldots, {\hat{r-2j-4-k}}, \ldots,
 \{{\hat{r-2j}}, \ldots, r+2j+3\}}\right)\cdot\left(\nu_{3,4 \ldots, {\hat{r-2j-2}}, 
\ldots , \{{\hat{r-2j+2}}, \ldots, r+2j+5\}}\right).$$

Applying Theorem~\ref{PS} for the pair of tuples
$$(1, 2, \ldots, {\hat{r-2j-k-4}}, \ldots, {\widehat{r-2j-2}}, r-2j-1, 
 \{{\hat{r-2j}}, \ldots, r+2j+3\})$$ and 
$(3, 4 \ldots, r-2j+1, \{{\hat{r-2j+2}}, \ldots, r+2j+5\})$
we get 
$$\begin{array}{l}
0 = (-1)^k
\left(\nu_{1,2, \ldots, {\hat{r-2j-4-k}},  \ldots,
  \{{\hat{r-2j-2}}, \ldots, r+2j+3\}, r-2j-4-k}\right)\cdot
\left( 
\nu_{3, 4 \ldots, {\hat{r-2j-4-k}}, \ldots,\{{\hat{r-2j+2}}, 
\ldots, r+2j+5\}}\right)\\\\
 -\left(\nu_{1,2 \ldots, {\hat{r-2j-4-k}}, \ldots  
\{{\hat{r-2j-2}}, \ldots, r+2j+3\}, r-2j-2}\right)\cdot
\left(\nu_{3,4 \ldots, {\hat{r-2j-2}}, \ldots,\{{\hat{r-2j+2}}, 
\ldots, r+2j+5\}}\right)\\\\
 
 +\left(\nu_{1,2 \ldots, {\hat{r-2j-4-k}},  \ldots, 
\{ {\hat {r-2j-2}}, \ldots, r+2j+3\}, r-2j}\right)\cdot
\left( 
\nu_{3,4 \ldots,  {\hat{r-2j}}, \ldots, \{{\hat{r-2j+2}},  \ldots, r+2j+5\}}\right)\\\\
-\left(\nu_{1,2 \ldots, {\hat{r-2j-4-k}}, \ldots, 
\{{\hat{r-2j-2}}, \ldots, r+2j+3\}, r+2j+5}\right)\cdot
\left(\nu_{3,4 \ldots, \{{\hat{r-2j+2}},  \ldots, r+2j+3\}}\right),
\end{array}$$
where the first term on the right hand side is $0$ if $k\in
\{r-2j-6, r-2j-5\}$. 
This implies 

\begin{multline}\label{BKj2}
C(j)C(j,5+k)^p- C(j, k+3)C(j,3)^p\\
 =  -\left(\nu_{1,2 \ldots, {\hat{r-2j-4-k}}, \ldots, {\hat{r-2j-2}}, 
 \ldots, 
\{ {\hat {r-2j+2}}, \ldots, r+2j+3\}}\right)\cdot C(j)^p\\
 +\left(\nu_{1,2 \ldots, {\hat{r-2j-4-k}}, \ldots, \{{\hat{r-2j-2}}, \ldots,
 r+2j+5\}}\right)C(j-1)^p.\end{multline}

 Adding terms from (\ref{BKj1}) and (\ref{BKj2}) in (\ref{BKj}), we get

 \begin{multline*}
  B_k^{(j+1)} = \tfrac{-1}{C(j)}
 \left(\nu_{1,2 \ldots, {\hat{r-2j-4-k}}, \ldots, {\hat{r-2j-2}}, \ldots, 
\{ {\hat {r-2j+2}}, \ldots, r+2j+3\}}\right)\\
+\tfrac{C(j-1)^p}{C(j)^{p+1}}
\left(\nu_{1,2 \ldots, {\hat{r-2j-4-k}}, \ldots, \{{\hat{r-2j-2}}, \ldots, r+2j+5\}}\right)\\
 +\tfrac{1}{C(j)}\left(\nu_{1,2 \ldots, {\hat{r-2j-4-k}}, \ldots, {\hat{r-2j-2}}, \ldots, 
\{ {\hat {r-2j+2}}, \ldots, r+2j+3\}}\right)\\
= \tfrac{C(j-1)^p}{C(j)^{p+1}}
\cdot C(j+1, k+1).\end{multline*}

\end{proof}

\section{A proof of the conjecture}
As we elaborated in Remark~\ref{r1}, to show that the set $\{X, f_1, \ldots, 
f_{s+1}, N_G(Z)\}$ is a SAGBI basis for ${\mathcal{K}}[X, Y, Z]^G$, it is enough to show that 
$\{X, f_1, \ldots, f_{s+2}\}$ is  a SAGBI basis for ${\mathcal{K}}[X, f_1,
\ldots, f_{s+2}]$, where $f_1, \ldots, f_{s+2}$ are 
 the elements as constructed in the previous two sections. Using 
Theorem~2.1, we first need to the find defining equations for the 
${\mathcal{K}}$-subalgebra $A={\mathcal{K}}[X, LT(f_1), \ldots, LT(f_{s+1}), LT(f_{s+2})] \subset 
{\mathcal{K}}[X, Y, Z]$.
Since $LT(f_i)$ is a power of $Y$ for $1\leq i\leq s+1$ and $LT(f_{s+2}) = Z^{p^r}$, it is enough to find the defining equations for 
${\mathcal{K}}[LT(f_1), \ldots, LT(f_{s+1})]$ which is a toric ring described by a numerical semigroup. 
We achieve this by using well known `gluing along a pair' operation repeatedly on numerical semigroups and  applying a result of
K.Watanabe \cite{Wa}.

First we recall few basic facts about numerical semigroups.

Let $a_1,\ldots, a_m$ be non-negative integers with $\gcd(a_1,\ldots,a_m)=1$. The set
\[
H=\{b_1a_1+\cdots +b_ma_m :\; b_i\in\mathbb{N} \text{ for } i=1,\ldots,m\}
\]
is an additively closed subset of the set  $\mathbb{N}$ of non-negative
integers.   It is called the numerical semigroup generated by
$a_1,\ldots,a_m$, denoted  $\langle a_1,\ldots,a_m\rangle$.
 
We fix a field $K$. The semigroup ring of $H$ is the $K$-subalgebra 
$K[H]$  of the polynomial ring $K[Y]$ which is generated by the elements  
$Y^{a_i}$ for  $i=1,\ldots,m$.

Let $T=K[t_1,\ldots,t_m]$ be the polynomial ring in the variables 
$t_1,\ldots,t_m$, and consider the $K$-algebra homomorphism
\[
\phi: T \rightarrow K[H],\quad t_i\mapsto Y^{a_i}\quad \text{for}\quad i=1,\ldots,m.
\]
We denote the kernel of $\phi$ by $I_H$. Since $K[H]$ is a toric ring, 
the kernel is generated by binomials. Furthermore, since $\phi$ is surjective, $K[H]\simeq T/I_H$.

We start with the above   numerical semigroup $H = \langle a_1,\ldots,a_m\rangle$,  
and let $(b,c)$ be a pair  of integers with $b\in H$, $c >1$ and 
$\gcd(b,c)=1$. The numerical semigroup  $H'=\langle b, cH\rangle 
=\langle b, ca_1,\ldots,ca_m\rangle$ is called the gluing   of $H$ 
with respect to $(b,c)$.  Since $b\in H$ we can write $b=\sum_{i=1}^mb_ia_i$ 
with $b_i\in\mathbb{N}$ for all $i$. Then $cb=\sum_{i=1}^mb_i(ca_i)$.

The ideal  $I_{H'}$  is the kernel of the $K$-algebra homomorphism

\[
\phi': T'=K[t_1,\ldots,t_{m+1}]\rightarrow K[H']
\]
with $\phi'(t_i)= Y^{ca_i}$ for $i=1,\ldots,m$ and $\phi'(t_{m+1})=Y^b$.

\begin{thm}{\em (\cite{Wa}, Lemma~1)}
\label{gluingrelation}
With the notation introduced we have
$$I_{H'}=(I_H,f)T', \quad\text{where}\quad f=t_{m+1}^c-\prod_{i=1}^mt_i^{b_i}.$$
\end{thm}

This theorem has the following nice  consequence.

\begin{cor}
\label{niceconsequence}
Let $H$ be the numerical semigroup which for $i=1,\ldots,r$ is obtained from
$\mathbb{N}$  by iterating  the gluing construction with respect to the
pairs of positive integers  $(b_i,c_i)$ with $\gcd(b_i,c_i)=1$ and $c_i>1$.
Then  $I_H$ is generated by binomials of the form
\[
t_{2}^{c_1}-u_1,\;  t_3^{c_2}-u_2,\; \ldots,\; t_{r+1}^{c_r}-u_r,
\]
where for $i=1,\ldots,r$, $u_i$ is a monomial in $K[t_1,\ldots,t_i]$.

In particular, $K[H]$ is a complete intersection.
\end{cor}

\begin{proof}
The first part of the statement follows from Theorem~\ref{gluingrelation}.
Note that $\mbox{height}~I_H$ is equal to the difference of  embedding dimension
and  dimension of $K[H]$. Thus,  $\mbox{height}~I_H =r+1-1=r$. Since $I_H$ is
generated by $r$ elements, $K[H]$ is a complete intersection.
\end{proof}

We now will apply Theorem~\ref{gluingrelation} and Corollary~\ref{niceconsequence}
to two special numerical semigroups, namely the ones  given by  
 $\{LT(f_1), \ldots, LT(f_{s+1})\}$.

\begin{lemma}\label{rodd}
Let $p>2$ be a prime number and $s$ a positive integer.  Recursively we
define the semigroups
\[H_0=\mathbb{N}, H_1=\langle 2,pH_0 \rangle,
H_i=\langle p^{2i-2}+2,  pH_{i-1} \rangle, i=2,\ldots,s.\]
Then
\begin{enumerate}
\item[(a)]For $s\geq 1$
$$H_s = \langle p^s, 2p^{s-1}, \{p^i+2p^{2s-2-i}\mid s\leq i\leq 2s-2\}\rangle.$$
\item[(b)]
$K[H_s]=K[t_1,\ldots, t_{s+1}]/I_{H_s}$ is a complete intersection with
\[
I_{H_s}=(p_1({\bf \underline{t}}),\ldots,p_s({\bf \underline{t}})),
\]
where $p_1({\bf \underline{t}})= t_2^2- t_1^p$,  $p_2({\bf \underline{t}}) = 
t_3^p-t_1t_2^p$
and $p_i({\bf \underline{t}}) = t_{i+1}^p-t_2^{p^{i-3}(p^2-1)}t_i$ for $i=3,\ldots,s$.
\end{enumerate}
\end{lemma}

\begin{proof}~(a)\quad
The assertion is  obvious 
for $s= 1$ and  $s=2$. The rest follows by induction on $s$.

\noindent~(b) {Claim}:\quad For each $1\leq i\leq s-1$, 
 the numerical semigroups $H_i$ is gluing of $H_{i-1}$ with respect to
$(p^{2i-2}+2, p)$.

We have $\gcd(p^{2i-2}+2, p) =1$. So we only need to show that 
$p^{2i-2}+2 \in H_i$ for $2\leq i\leq s$.

  Indeed, we have

(i) $2\in H_0$;\\

(ii) $p^2+2= 1\cdot 2+ p\cdot p$ with $2,p\in H_1$;\\

(iii) $p^4+2= (p^2-1)\cdot p^2+1\cdot (p^2+2)$ with $p^2,p^2+2\in H_2$;\\

(iv) for $i=3,\ldots, s-1$,  $p^{2(i+1)-2}+2=(p^i-p^{i-2})\cdot p^i+1\cdot
(p^{2i-2}+2)$

\hspace{0.8cm}with
$p^i, p^{2i-2}+2\in H_i$.

This proves the claim, and hence $K[H_s]$ is a complete intersection
by Corollary~\ref{niceconsequence}. Finally, Theorem~\ref{gluingrelation}
together with (i)-- (iv)
give us the generators of $I_{H_s}$.
\end{proof}

The second special case to be considered is given in

\begin{lemma}\label{reven}
Let $p>3 $ be a prime number and $s$ a positive integer.  Recursively we
define the semigroups
\[
H_0=\mathbb{N},
H_i=\langle p^{2i-1}+2,  pH_{i-1} \rangle, i=1,\ldots,s.
\]
Then
\begin{enumerate}
\item[(a)] For $s\geq 1$
$$H_s=\langle p^s, \{p^i+2p^{2s-1-i}\mid s\leq i \leq 2s-1\}\rangle $$
\item[(b)]
$K[H_s]=K[t_1,\ldots, t_{s+1}]/I_{H_s}$ is a complete intersection with
\[
I_{H_s} = (p_1({\bf \underline{t}}),\ldots, p_s({\bf \underline{t}})),
\]
where $p_1({\bf \underline{t}}) =  t_2^p-t_1^{p+2}$
and $p_i({\bf \underline{t}})=t_{i+1}^p-t_1^{p^{i-2}(p^2-1)}t_i$ for $i=2,\ldots,s$.
\end{enumerate}
\end{lemma}

\begin{proof}
The proof proceeds in the same way as that of Proposition~\ref{rodd}.
\end{proof}
\vspace{10pt}

Here we follow the Notation~\ref{n2}.

\begin{thm}\label{t1}Let  $r$ be an odd integer such that $r = 2s-1\geq 3$. 

Then there exists homogeneous polynomials $f_1, f_2, 
\ldots, f_{s+1}\in {\mathcal{K}}[X, Y, Z]^G$ 
such that 
\begin{enumerate}
\item[(1)] $\{X, f_1, f_2, \ldots, f_{s+1}, N_G(Z)\}$ is a SAGBI basis for 
${\mathcal{K}}[X, Y, Z]^G$. 
\item[(2)] The ring 
${\mathcal{K}}[X, Y, Z]^G$ is a complete intersection ring.
\end{enumerate} 
Moreover 
$$LT({\mathcal{K}}[X, Y, Z]^G) = {\mathcal{K}}[X, Y^{2p^{s-1}}, Y^{p^s}, \{Y^{p^{s+i-3}+2p^{s-i+1}}\}_{3\leq i\leq s}, Z^{p^r}]$$
and
$${\mathcal{K}}[X, Y, Z]^G = \frac{{\mathcal{K}}[t_0, t_1, \ldots, t_{s+2}]}{(q_1({\bf\underline{ t}}), \ldots,
q_s({\bf\underline{ t}}))},$$
where 
 $$\begin{array}{lcl}
q_1({\bf\underline{ t}}) & = & F_3({\bf\underline{ t}}) +
\tfrac{2b_2}{a_0}t_3 t_0^{p^s-2p^{s-2}},\\\\
  q_{j+2}({\bf\underline{ t}}) & = & F_{j+4}({\bf \underline{t}}) -\tfrac{B(j-1)^pB(j+1)}{B(j)^{p+1}}
t_0^{2p^{s-1-j}-2p^{s-3-j}} t_{j+4},\quad\mbox{for}\quad  0\leq j\leq s-2.
\end{array}$$
 
\end{thm}

\begin{proof}Let $f_1, f_1. \ldots, f_{s+2}\in {\mathcal{K}}[X, Y, Z]^G$ 
be as in Proposition~\ref{odd}. Note that ${\mathcal{K}}[X, f_1, f_2] = {\mathcal{K}}[X, g_1, g_2]$, where 
$g_1$, $g_2$ are the elements as in Theorem~\ref{CSW2}.

\vspace{5pt}

\noindent~(1).\quad Let $f_0 = X$. 
By Remark~\ref{r1}, it is enough to prove that  
$\sB_1 = \{f_0, f_1, \ldots, f_{s+1}, f_{s+2}\}$ is a 
SAGBI basis for $A = {\mathcal{K}}[f_0, f_1, \ldots, f_{s+2}]$.

Let  ${\tilde  A} = {\mathcal{K}}[LT(f_0), LT(f_1) \ldots, LT(f_{s+2})]$ 
then  ${\tilde  A} = {\mathcal{K}}[H_s][X, Z^{p^r}]$ where 
$H_s$ is as in Lemma~\ref{rodd}, and 
 the map
$$\phi: {\mathcal{K}}[{\bf \underline{t}}] = {\mathcal{K}}[t_0, t_1, \ldots, t_{s+2}]
\longrightarrow {\tilde A}\quad\mbox{given by}\quad 
t_i \to LT(f_i),$$ has 
$$\mbox{Ker}~\phi = \left(p_1({\bf\underline{ t}}), 
p_2({\bf \underline{t}}), \ldots, p_s({\bf\underline{ t}})\right),$$ 
where
$p_1({\bf \underline{t}})  = t_2^2-t_1^p,\quad  p_2({\bf\underline{ t}}) = 
t_3^p- t_1t_2^p$ 
and   
$$p_{j+2}({\bf\underline{ t}}) =  t^p_{j+3} -  t_2^{p^{j-1}(p^2-1)}t_{j+2}, 
\quad\mbox{for}\quad 1\leq j \leq s-2.$$

By Theorem~\ref{RSKM}, it is enough to show that 
the pairs
\begin{multline}\label{subduct}
(f_2^2,~~f_1^p),\quad    
(f_3^p- f_1f_2^p)\quad\mbox{and}\quad   
 (f^p_{j+3} -  f_2^{p^{j-1}(p^2-1)}f_{j+2}), 
\quad\mbox{for}\quad 1\leq j \leq s-2
\end{multline}

subducts to $0$ in $A$.

Now (\ref{f3}) gives  

\begin{equation}\label{p1f} p_1({\bf\underline{ f}}) = -\left(\frac{2b_1}{a_0}\right)\cdot f_1^{(p+1)/2}X^{p^s-p^{s-1}} -
\left(\frac{2b_2}{a_0}\right)
f_3 X^{p^s-2p^{s-2}},\end{equation}

 where $LT(f_1^{(p+1)/2}X^{p^s-p^{s-1}}) > 
LT(f_3 X^{p^s-2p^{s-2}})$.

Evaluating (\ref{F4t}) and (\ref{Fj+4t}) at ${\underline{f}}$,
for $0\leq j\leq s-2$, we    get
\begin{multline}\label{pj+2} p_{j+2}({\bf\underline{ f}})   =   
 - \left(N_{1j}\right)\cdot E_{1j}({\bf \underline{f}})X^{p^{s-1-j}}
-\left(L_{2j}\right)\cdot  f_2^{p^{j}(p-1)}f_{j+3}X^{2p^{s-1-j}-2p^{s-2-j}}\\\\
 - \left(N_{2j}\right)\cdot E_{2j}({\bf\underline{ f}})({\bf \underline{f}})X^{2p^{s-1-j}-p^{s-2-j}} + 
\left(\tfrac{B(j-1)^pB(j+1)}{B(j)^{p+1}}\right)f_{j+4} X^{2p^{s-1-j}-2p^{s-3-j}},\end{multline}
where 
\begin{multline*}
LT(E_{1j}({\bf\underline{ f}})X^{p^{s-1-j}}) > 
LT(f_2^{p^{j}(p-1)}f_{j+3}X^{2p^{s-1-j}-2p^{s-2-j}})\\
 > LT(E_{2j}({\bf \underline{f}})X^{2p^{s-1-j}-p^{s-2-j}}) > 
LT(f_{j+4} X^{2p^{s-1-j}-2p^{s-3-j}}).\end{multline*}
 
 In particular every pair as given in (\ref{subduct}) subducts to zero in $A$.

\vspace{5pt}

\noindent~(2).\quad  Let $J$ be the kernel of the 
${\mathcal{K}}$-algebra homomorphic
$$\varphi: R={\mathcal{K}}[t_0,\ldots, t_{s+2}]\longrightarrow A\quad\mbox{given by}\quad
t_0\mapsto X,~~~ t_i\mapsto f_i,$$
for $i=1,\ldots, s+2$.

Since $\sB_1 = \{X,f_1,\ldots,f_{s+1}, f_{s+2}\}$ is a 
SAGBI basis for $B ={\mathcal{K}}[X,Y,Z]^G$, the relations generating $J$ come
from the  subductions of the t\^{e}te-\'{a}-t\^{e}tes
$(f_2^2,  f_1^p)$,  $(f_3^p, f_1f_2^p)$
and $(f_{i+1}^p, f_2^{p^{i-3}(t^2-1)}f_i)$ for $i=3,\ldots,s$, 
which result  from   the binomial relations $p_1({\bf\underline{ t}}), p_2({\bf\underline{ t}}), 
\ldots, p_s({\bf\underline{ t}})$. 
Since 
$$ F_3({\bf\underline{ t}}) =  p_1({\bf\underline{ t}}) + \left(\tfrac{2b_1}{a_0}\right)\cdot t_1^{(p+1)/2}
t_0^{p^s-p^{s-1}} $$
and for $0\leq j\leq s-2$

\begin{multline*}
F_{j+4}({\bf \underline{t}}) = p_{j+2}({\bf \underline{t}}) +   
  \left(N_{1j}\right)\cdot E_{1j}({\bf \underline{t}})t_0^{p^{s-1-j}}\\
+ \left(L_{2j}\right)\cdot  t_2^{p^{j}(p-1)}t_{j+3}t_0^{2p^{s-1-j}-2p^{s-2-j}}
+ \left(N_{2j}\right)\cdot E_{2j}({\bf\underline{ t}})({\bf\underline{ t}})
t_0^{2p^{s-1-j}-p^{s-2-j}},
\end{multline*}
it
follows, by (\ref{p1f}) and (\ref{pj+2}) that $J$ is generated $q_1({\bf\underline{ t}}), \ldots, q_s({\bf\underline{ t}})$.
 Since  the Krull
dimension of $A$ is $3$ and its  embedding  dimension is $s+3$,
it follows that $\mbox{height}~J=s$. This proves that $A$ is a
complete intersection.

\end{proof}

Here we follow the Notation~\ref{ne}.

\begin{thm}\label{t1e}Let  $r=2s\geq 4$ be an even integer  and $p>3$ be a 
prime number.

Then there exists homogeneous polynomials $f_1, f_2, 
\ldots, f_{s+1}\in {\mathcal{K}}[X, Y, Z]^G$ 
such that 
\begin{enumerate}
\item[(1)] $\{X, f_1, f_2, \ldots, f_{s+1}, N_G(Z)\}$ is a SAGBI basis for 
${\mathcal{K}}[X, Y, Z]^G$. 
\item[(2)] The ring 
${\mathcal{K}}[X, Y, Z]^G$ is a complete intersection ring.
\end{enumerate} 
Moreover 
$$LT({\mathcal{K}}[X, Y, Z]^G) = {\mathcal{K}}[X, Y^{p^s}, 
\{Y^{p^{s+i-2}+2p^{s-i+1}}\}_{2\leq i\leq s}, Z^{p^r}]$$
and

$${\mathcal{K}}[X, Y, Z]^G = \frac{{\mathcal{K}}[t_0, t_1, \ldots, t_{s+2}]}{(q_1({\bf\underline{ t}}), \ldots,
q_s({\bf \underline{t}}))},$$
where  
 $$\begin{array}{lcl}
q_1({\bf \underline{t}}) & =  & F_3({\bf \underline{t}}) -\tfrac{c_0^pC(0)}{c_1^{p+1}}t_3 
t_0^{2p^s-2p^{s-2}},\\\\
q_{j+2}({\bf\underline{ t}}) & = & F_{j+4}({\bf\underline{ t}}) -\tfrac{C(j-1)^pC(j+1)}{C(j)^{p+1}}
t_0^{2p^{s-1-j}-2p^{s-3-j}} t_{j+4},\quad\mbox{for}\quad  1\leq j\leq s-2.\end{array}$$
 \end{thm}
\begin{proof}Here we choose $f_0 = X$ and $f_1, \ldots, f_{s+2}$ 
as in Proposition~\ref{even}.
Let  ${\tilde  A} = {\mathcal{K}}[LT(f_0=X), LT(f_1) \ldots, LT(f_{s+2})]$ 
then  ${\tilde  A} = {\mathcal{K}}[H_s][X, Z^{p^r}]$ where 
$H_s$ is as in Lemma~\ref{reven}, and 
then  the map
$$\phi: {\mathcal{K}}[{\bf\underline{ t}}] = {\mathcal{K}}[t_0, t_1, \ldots, t_{s+2}]
\longrightarrow {\tilde A}\quad\mbox{given by}\quad 
t_i \to LT(f_i),$$ has 
$$\mbox{Ker}~\phi = \left(p_1({\bf\underline{ t}}), p_2({\bf\underline{ t}}), 
\ldots, p_s({\bf\underline{ t}})\right),$$ 
where
$$p_1({\bf\underline{ t}})  = t_2^p-t_1^{p+2},\quad     
p_{j+2}({\bf\underline{ t}}) =  t^p_{j+3} -  t_1^{p^j(p^2-1)}t_{j+2}, 
\quad\mbox{for}\quad 1\leq j \leq s-2.$$
Now the proof is along the same lines as for $r$ equal to the odd case.
\end{proof}

\end{document}